\documentclass{article}
\usepackage{graphicx} 
\usepackage{amssymb, amsmath, amsthm}

 \newtheorem{thm}{Theorem}[section]
 \newtheorem{cor}[thm]{Corollary}
 \newtheorem{lem}[thm]{Lemma}
 \newtheorem{prop}[thm]{Proposition}

    \theoremstyle{definition}
 \newtheorem{defn}[thm]{Definition}

 \newtheorem{exam}[thm]{Example}
 \theoremstyle{remark}
 \newtheorem{rem}[thm]{Remark}
\numberwithin{equation}{section}
\title{Structures and comodules of Hom-post Lie
coalgebras }

\author{Houndedji Damien, \\ D\'epartement de Math\'ematiques, Universit\'e Gamal Abdel Nasser de Conakry,\\
BP : 1147, Conakry, Guinea.\\
houndedjidamien@gmail.com\\
and\\
Ibrahima Bakayoko,\\ D\'epartement de Math\'ematiques, Universit\'e de N'Z\'er\'ekor\'e,\\
BP : 50, N'Z\'er\'ekor\'e, Guinea.\\
ibrahimabakayoko27@gmail.com
}

\date{}

\begin{document}

\maketitle
\begin{abstract}
In this paper, we introduce the notions of Hom-tridendriform coalgebras and Hom-post-Lie coalgebras as the
dual notions of Hom-tridendriform algebras and Hom-post-LIe algebras respectively. We give some properties realted to them. Then, we study the relationships between them and their connection with post-Hom-Poisson coalgebras. Next, using the Yau' stwisting in the modules case, we give some constructions of  comodules over post-Hom-Lie coalgebras by twisting either the comodule structures or post-Hom-Lie coalgebra structures.
\end{abstract}

\section{Introduction}
Post-Lie algebras first arise from the work of Bruno Vallette \cite{BV} in 2007 through the purely operadic technique of Koszul dualization. In Hom-setting, they are the twisted version of post-Lie algebras \cite{BV}, \cite{BOL}, \cite{XC}, \cite{YC},
 which are both Hom-Lie algebras and commutative dendriform algebras \cite{JL1}, \cite{JL2}.
The two structures being connected by some compatibility conditions.
In \cite{ML}, it shown that they also arise naturally from differential geometry of homogeneous spaces and Klein geometries, topics that are
closely related to Cartan's method of moving frames. The universal
enveloping algebras of post-Lie algebras and the free post-Lie algebra are studied. Some examples and related structures are given.
Among other results, the authors proved in \cite{CLX} that tridendriform and Rota-Baxter Lie algebras lead to the structures of
post-Lie algebras.

Given a classical algebraic structure, the study of its dual (coalgebraic structure) appears some decades after the classical one. For example, Cohomology of Lie algebras was introduced over 70 years ago by C. Chevalley and S. Eilenberg
and given a standard textbook treatment in the classical treatise of H. Cartan and S. Eilenberg
[7]. Nowadays, the study of  invariants of coalgebras has received the attention of mathematicians as a part of the study of the dual to
algebras and for their many special properties. For instance, construction of funtors between Lie algebras and Lie coalgebras \cite{L}; a conformal analog of the anti-equivalence correspondence between the category of finite free differential Lie (super)coalgebras and the category of finite free Lie conformal (super)algebras and
the classification of simple finite-rank differential Lie
(super)coalgebras \cite{BL}; coassociative coalgebras and comodules and their homological algebra \cite{LP}, Koszul correspondence between Lie coalgebras and commutative algebras \cite{AL}; Rota-Baxter coalgebra \cite{R, S}; pre-Lie coalgebras \cite{S}; Dendriform coalgebras \cite{AD}, Hom-(co)algebras \cite{DY}, Hom-(co)modules and so on.

Inspired by our previous former works \cite{B1, B2} and the above works on certain coalgebras, we introduce a
post-Hom-Lie coalgebra as a Hom-Lie coalgebra $(C, \gamma, \alpha)$ together with a linear comultiplication $\Delta : C\rightarrow C\otimes C$ such that :
\begin{eqnarray}
(\alpha\otimes\gamma)\circ\Delta-(\Delta\otimes \alpha)\circ\gamma-(\tau\otimes I)\circ(\alpha\otimes\Delta)\circ\gamma=0,\hspace{4cm}\\
(\Delta\otimes \alpha)\circ\Delta\!-\!(\alpha\otimes\Delta)\circ\Delta\!-\!(\tau\otimes I)\circ(\Delta\otimes \alpha)\circ\Delta\!
+\!(\tau\otimes I)\circ(\alpha\otimes\Delta)\circ\Delta\!+\!(\gamma\otimes \alpha)\circ\Delta=0.
\end{eqnarray}

In the present paper, we aim to
introduce and study post-Hom-Lie coalgebras as the dual structure of post-Hom-Lie algebras. The paper is organized as follows.
In section 2, we recall definitions of Hom-coassociative coalgebras, Hom-Lie coalgebras and Hom-preLie coalgebras, and some elementaries properties of these Structures.
 In section 3, we introduce Hom-tridendriform coalgebras, give some constructions by twisting, make them in connection with Hom-associative Rota-baxter coalgebras and Hom-preli coalgebras. Then we introduce their comodules and a give a twisting result. In section 4, we introduce post-Hom-Lie coalgebras and give constructions of new structures from either a given one, an endomorphism or tensor product. Then we study their admissibitlity and their relationship with Hom-pre-Lie coalgebras and Hom-Lie Rota-Baxter coalgebras. In order to study the connection of post-Hom-Lie coalgebra and cocommutative Hom-Poisson coalgebras, we introduce  commutative Hom-tridendriform coalgebras and Post-Hom-Poisson coalgebras. In the last section, we introduction of comodules over post-Hom-Lie coalgebras and the given of various constructions from tensor product or by twisting the structures of either post-Hom-Lie or comodule.

In this paper, we work over a field $\mathbb{F}$ of characteristic $0$.
\section{Preliminaries}
In this section, we recall definition and basic properties of Hom-coassociative coalgebras, Hom-Lie coalgebras and Hom-preLie coalgebras.
\begin{defn}
 A Hom-coassociative coalgebra is a triple $(A,\Delta_\cdot,\alpha)$ consisting of a vector space $A$, a $\mathbb{K}$-bilinear map
$\Delta_\cdot : A\longrightarrow A\otimes A $ and a linear map $\alpha :A\longrightarrow A $ satisfying 
\begin{eqnarray}
(\alpha\otimes\Delta_\cdot)\circ\Delta_\cdot=(\Delta_\cdot\otimes\alpha)\circ\Delta_\cdot\quad\mbox{(coassociativity)}, \label{coasso}.
\end{eqnarray}
If in addition $\alpha$ satisfies
\begin{eqnarray}
\Delta_\cdot\circ\alpha=(\alpha\otimes\alpha)\circ\Delta_\cdot\quad\mbox{(comultiplicativity)}, \label{multip}
\end{eqnarray}
then $(A,\Delta_\cdot,\alpha)$ is said to be comultiplicative.
\end{defn}
When $\alpha=Id_A$, $(A,\Delta, Id_A)$, simply denoted $(A,\Delta)$, is a coassociative coalgebra.\\

 The lemma below allows to get a Hom-coassociative coalgebra from an coassociative coalgebra and an coalgebra endomorphism.
\begin{lem}\label{ha1}
 Let $(A,\Delta)$ be a coassociative coalgebra and $\alpha: A\longrightarrow A$ be a coalgebra endomorphism.
Then, the triple $(A,\Delta_\alpha,\alpha)$, where $\Delta_\alpha=\Delta\circ\alpha$, is a comultiplicative Hom-coassociative coalgebra.
\end{lem}
\begin{defn}
A Hom-coassociative Rota-Baxter coalgebra of weight $\lambda\in\mathbb{K}$ is a Hom-coassociative
coalgebra $(A, \Delta_\cdot, \alpha)$ together with a linear map $R : A\rightarrow A$ that satisfies the identities
\begin{eqnarray}
R\circ\alpha&=&\alpha\circ R,\\
(R\otimes R)\circ\Delta &=& \Big((R\otimes I)\circ\Delta + (I\otimes R)\circ\Delta +\lambda \Delta\Big)\circ R.
\end{eqnarray}
\end{defn}

\begin{defn}\label{hplad}
A Hom-Lie coalgebra is a triple $(L,\gamma, \alpha)$ consisting of a vector space $L$, a linear map
 $\gamma : L\longrightarrow L\otimes L$ and a linear map $\alpha :L\longrightarrow L$ satisfying 
 \begin{eqnarray}
  \gamma=-\tau\circ\gamma\quad \mbox{(skew-cocommutativity)},\\
 \gamma\circ\alpha= \alpha^{\otimes^{2}}\circ\gamma\quad \mbox{(comultiplicativity)},\\
(1+\xi+\xi^2)\circ(\alpha\otimes\gamma)\circ\gamma=0\quad \mbox{(Hom-co-Jacobi identity)},
 \end{eqnarray}
where $\tau(x\otimes y)=y\otimes x$, $\xi(x\otimes y\otimes z)=y\otimes z\otimes x$ and $\xi^2(x\otimes y\otimes z)=z\otimes x\otimes y$, for all $x, y, z\in L$.
\end{defn}
When $\alpha=Id_L$, we obtain the definition of Lie coalgebras.\\

The following result is the co-Lie-version of Lemma \ref{ha1}.
\begin{prop}\label{Amb}
 Let $(L,\gamma)$ be a Lie coalgebra and $\alpha$ be a Lie coalgebra endomorphism.  Then, $(L,\gamma_\alpha,\alpha)$
 is a Hom-Lie coalgebra, where $\gamma_\alpha=\gamma\circ\alpha$.
\end{prop}

The following Lemma connects Hom-coassociative coalgebras to Hom-Lie coalgebras i.e. to any Hom-coassociative
coalgebra $A$ one may associate a Hom-Lie coalgebra.
\begin{lem}\label{Am1}
 Let $(A,\Delta,\alpha)$ be a Hom-coassociative coalgebra. Then, $(A, \gamma,\alpha)$ is a Hom-Lie coalgebra, where
 $\gamma=(1-\tau)\circ\Delta$.
\end{lem}
\begin{defn}
 Let $(A, \Delta_\cdot, \alpha)$ be a Hom-coalgebra structure on $V$ defined by the comultiplication $\Delta_\cdot$ and the coalgebra Homomorphism $\alpha$. Then, $A$ is said to
be Hom-admissible coalgebra over $V$ if the bracket defined  by
\begin{eqnarray}
 \gamma = (1-\tau)\circ\Delta_\cdot \label{lia}
\end{eqnarray}
satisfies the Hom-co-Jacobi identity $ (1+\xi+\xi^2)\circ(\alpha\otimes\gamma)\circ\gamma=0$.
\end{defn}
\begin{rem}
Since the commutator bracket (\ref{lia}) is always skew-cocommutative, it makes any Hom-Lie admissible coalgebra into a Hom-Lie coalgebra.
\end{rem}

\begin{defn}
 A Hom-preLie coalgebra is a triple $(A, \Delta_\cdot, \alpha)$ in which $\Delta_\cdot :  A\rightarrow A\otimes A$ is a linear map, $\alpha : A\rightarrow A$
 is a linear map and
\begin{eqnarray}
(\Delta_\cdot\otimes\alpha)\circ\Delta_\cdot-(\alpha\otimes\Delta_\cdot)\circ\Delta_\cdot=(\tau\otimes I)\circ\Big((\Delta_\cdot\otimes\alpha)\circ\Delta_\cdot-(\alpha\otimes\Delta_\cdot)\circ\Delta_\cdot\Big).
\end{eqnarray}
\end{defn}
\begin{prop}\cite{IB}
 Let $(C, \Delta_C, \alpha_C)$ be a Hom-coassociative coalgebra and
$\gamma_C: C\longrightarrow C\otimes C$ be a linear map defined by
$$\gamma_C(x)=x_1\otimes x_2-x_2\otimes x_1\quad \mbox{with}\quad \Delta_C(x)=x_1\otimes x_2.$$
Then $(L(C), \gamma_C, \alpha_C)$ is a Hom-Lie coalgebra, where $L(C)=C$ as vector space.
\end{prop}
\section{Hom-tridendriform coalgebras and comodules}
 We introduce Hom-tridendriform coalgebras and their comodules, study their twisting and give their relationship with
 Hom-coassociative (Rota-Baxter) coalgebras and Hom-preLie coalgebras.
\subsection{Hom-tridendriform coalgebras}
\begin{defn}\label{htca}
A Hom-tridendriform coalgebra is a vector space $C$ together with three linear maps $\Delta_{-1}, \Delta_{0}, \Delta_{1} : C\rightarrow C$ and
 a linear map $\alpha : A\rightarrow A$ satisfing the  following relations :
 \begin{eqnarray}
   (\Delta_{-1}\otimes\alpha)\circ\Delta_{-1}&=&(\alpha\otimes\Delta_{-1}+\alpha\otimes\Delta_1+\alpha\otimes\Delta_0)\circ\Delta_{-1},\label{c1}\\
(\Delta_1\otimes\alpha)\circ\Delta_{-1}&=&(\alpha\otimes\Delta_{-1})\circ\Delta_1,\label{c2}\\
(\alpha\otimes\Delta_1)\circ\Delta_1&=&(\Delta_{-1}\otimes\alpha+\Delta_1\otimes\alpha+\Delta_0\otimes\alpha)\circ\Delta_1,\label{c3}\\
 (\Delta_{-1}\otimes\alpha)\circ\Delta_0&=&(\alpha\otimes\Delta_1)\circ\Delta_0,\label{c4}\\
 (\Delta_1\otimes\alpha)\circ\Delta_0&=&(\alpha\otimes\Delta_0)\circ\Delta_1,\label{c5}\\
 (\Delta_0\otimes\alpha)\circ\Delta_{-1}&=&(\alpha\otimes\Delta_{-1})\circ\Delta_0,\label{c6}\\
 (\Delta_0\otimes\alpha)\circ\Delta_0&=&(\alpha\otimes\Delta_0)\circ\Delta_0.\label{c7}
 \end{eqnarray}
 \end{defn}
 \begin{exam}
 $(T, \Delta_\dashv, \Delta_\vdash, \Delta_\cdot, \alpha)$ is a Hom-tridendriform coalgebra if and only if \\
 $T^{op}=(T, \Delta_\dashv^{op}, \Delta_\vdash^{op}, \Delta_\cdot^{op}, \alpha)$  is, where
$$\Delta_\dashv^{op}:=\tau\circ\Delta_\vdash,\quad \Delta_\vdash^{op}:=\tau\circ\Delta_\dashv \quad\mbox{and}\quad \Delta_\cdot^{op}:=\tau\circ\Delta_\cdot.$$
\end{exam}
 \begin{rem}
  Hom-tridendriform coalgebras contain Hom-coassociative coalgebras \cite{MS} ($\Delta_{-1}=\Delta_1=0$) and Hom-dendriform coalgebras
($\Delta_0=0$) as special cases.
 \end{rem}
 \begin{prop}
 Let  $(T, \dashv, \vdash, \cdot, \alpha)$ be a finite dimensional Hom-tridendriform algebra. Then, its dual space $T^*$ is provided with a
 Hom-tridendriform coalgebra
 $(T^*, \dashv^*, \vdash^*, \cdot^*, \alpha^*)$, where $ \dashv^*, \vdash^*, \cdot^*, \alpha^*$  are the transpose maps.
\end{prop}
\begin{defn}
 A morphism $f : (C, \Delta_{-1}, \Delta_0, \Delta_1, \alpha)\rightarrow (C', \Delta'_{-1}, \Delta'_0, \Delta'_1, \alpha')$
 of Hom-tridendriform coalgebras is a linear map of the underline vector spaces such that
 $$f\circ\alpha=\alpha'\circ f,\quad\Delta_{-1}'\circ f=(f\otimes f)\circ\Delta_{-1},
 \quad\Delta_{0}'\circ f=(f\otimes f)\circ\Delta_0,\quad\Delta_1'\circ f=(f\otimes f)\circ\Delta_1.$$
\end{defn}
 \begin{defn}
 A Hom-tridendriform coalgebra $(C, \Delta_{-1}, \Delta_0, \Delta_1, \alpha)$ in which $\alpha$ is an endomorphism, with respect to
 all the products, is said to be multiplicative. 
 \end{defn}

\begin{thm}
  Let $(C, \Delta_{-1}, \Delta_0, \Delta_1, \alpha)$  be a Hom-tridendriform coalgebra and $\beta : C\rightarrow C$ be an  endomorphism.
 Then,
$(C, \Delta^\beta_{-1}=\Delta_{-1}\circ\beta, \Delta^\beta_0=\Delta_0\circ\beta, \Delta^\beta_1=\Delta_1\circ\beta, \beta\circ\alpha)$ 
  is a Hom-tridendriform coalgebra.
 \end{thm}
 \begin{proof}
 We only prove (\ref{c2}). We have 
 \begin{eqnarray}
 (\Delta^\beta_{1}\otimes\beta\circ\alpha)\circ\Delta^\beta_{-1}
 &=&(\Delta_{1}\circ\beta\otimes\beta\circ\alpha)\circ\Delta_{-1}\circ\beta\nonumber\\
&=&\beta^{\otimes3}\circ(\Delta_{1}\otimes\alpha)\circ(\beta\otimes\beta)\circ\Delta_{-1}\nonumber\\
&=&\beta^{\otimes3}\circ(\Delta_{1}\circ\beta\otimes\alpha\circ\beta)\circ\Delta_{-1}\nonumber\\
 &=&(\beta^2)^{\otimes3}\circ(\Delta_{1}\otimes\alpha)\circ\Delta_{-1}\nonumber\\
 &=&(\beta^2)^{\otimes3}\circ(\alpha\otimes\Delta_{-1})\circ\Delta_1\quad\mbox{(by (\ref{c2})})\nonumber\\
 &=&(\beta^2\circ\alpha\otimes((\beta^2\otimes\beta^2)\circ\Delta_{-1})\circ\Delta_1\nonumber\\
 &=&(\beta\circ\alpha\circ\beta\otimes\Delta_{-1}\circ\beta^2)\circ\Delta_1\nonumber\\
 &=&(\beta\circ\alpha\otimes\Delta_{-1}\circ\beta)\circ(\beta\otimes\beta)\circ\Delta_1\nonumber\\
 &=&(\beta\circ\alpha\otimes\Delta_{-1}\circ\beta)\circ\Delta_1\circ\beta\nonumber\\
 &=&(\beta\circ\alpha\otimes\Delta^\beta_{-1})\circ\Delta^\beta_1\nonumber.
 \end{eqnarray}
 The other axioms are proved similarly.
 \end{proof}
\begin{cor}
 If $(C, \Delta_{-1}, \Delta_0, \Delta_1, \alpha)$ is a multiplicative Hom-tridendriform coalgebra, then so is
$$(C, \Delta^n_{-1}=\Delta_{-1}\circ\alpha^n, \Delta^n_0=\Delta_0\circ\alpha^n, \Delta^n_1=\Delta_1\circ\alpha^n, \alpha^{n+1}).$$
 \end{cor}
\begin{cor}
 Let $(C, \Delta_{-1}, \Delta_0, \Delta_1, \alpha)$ be a multiplicative Hom-tridendriform coalgebra such that $\alpha$ be inversible.
 Then 
 $$(C, \Delta_{-1}\circ\alpha^{-1}, \Delta_0\circ\alpha^{-1}, \Delta_1\circ\alpha^{-1})$$ 
 is a tridendriform coalgebra.
 \end{cor}
 The below result associates a Hom-coassociative coalgebra to a given Hom-tridendriform coalgebra.
 \begin{thm}
 Let $(C, \Delta_{-1}, \Delta_0, \Delta_1, \alpha)$  be a Hom-tridendriform coalgebra. Define the comultiplication
 $\Delta : C\rightarrow C\otimes C$ as
 $$\Delta=\Delta_{-1}+\Delta_0+\Delta_1.$$
 Then $(C, \Delta, \alpha)$ is a Hom-coassociative coalgebra.
\end{thm}
 \begin{proof}
  By extending $c_\alpha(\Delta):=(\Delta\otimes\alpha)\circ\Delta-(\alpha\otimes\Delta)\circ\Delta$ by means of $\Delta_i, i=-1, 0, 1$ 
 and using axioms in Definition \ref{htca}, it follows that $(C, \Delta, \alpha)$ is a Hom-coassociative coalgebra.
 \end{proof}
 The following theorem asserts that any Hom-coassociative Rota-Baxter coalgebra give rise to Hom-tridendriform coalgebra.
\begin{thm}
  Let $(C, \Delta_0, \alpha, R)$ be a Hom-coassociative Rota Baxter coalgebra of weight $\lambda\in\mathbb{K}$.
 Let us set 
 $$ \Delta_{-1}=(I\otimes R)\circ\Delta_{0},\quad \Delta_{1}=(R\otimes I)\circ\Delta_0,\quad \bar{\Delta}_0=\lambda\Delta_0.$$
 Then, $(C, \Delta_{-1}, \bar{\Delta}_0, \Delta_1, \alpha)$ is a Hom-tridendriform coalgebra.
 \end{thm}
 \begin{proof}
  Here also, we only prove (\ref{c2}). We have 
 \begin{eqnarray}
  (\Delta\otimes\alpha)\circ\Delta_{-1}
 &=&((R\otimes I)\circ\Delta_0\otimes\alpha)\circ((I\otimes R)\circ\Delta_0)\nonumber\\
 &=&((R\otimes I)\circ\Delta_0\otimes\alpha\circ R)\circ\Delta_0\nonumber\\
 &=&((R\otimes I)\circ\Delta_0\otimes R\circ\alpha)\circ\Delta_0\nonumber\\
 &=&((R\otimes I\otimes R)\circ(\Delta_0\otimes\alpha)\circ\Delta_0\nonumber\\
 &=&((R\otimes I\otimes R)\circ(\alpha\otimes\Delta_0)\circ\Delta_0\nonumber\\
 &=&(R\circ\alpha\otimes(I\otimes R)\circ\Delta_0)\circ\Delta_0\nonumber\\
 &=&(\alpha\circ R\otimes \Delta_{-1})\circ\Delta_0\nonumber\\
 &=&(\alpha\otimes\Delta_{-1})\circ(R\otimes I)\circ\Delta_0\nonumber\\
 &=&(\alpha\otimes\Delta_{-1})\circ\Delta_1.\nonumber
 \end{eqnarray}
 Similarly we prove the other relations.
 \end{proof}
 
 \begin{cor}
  Let $(C, \Delta_0, \alpha, R)$ be a Hom-coassociative Rota Baxter coalgebra of weight $\lambda\in\mathbb{K}$. Then,
 $(C, \bar\Delta_{-1}, \bar\Delta_1, \alpha)$ is a Hom-dendriform coalgebra, where
 $$\bar\Delta_{-1}=\Delta_{-1}+\lambda\Delta_{0}\quad\mbox{and}\quad \bar\Delta_{1}=\Delta_{1}.$$
 \end{cor}
 
 Now let $\mathcal{S}_3$ be the symmetric group of order $3$. Given $\sigma\in\mathcal{S}_3$, we define a linear map 
 $$\Phi_\sigma : C^{\otimes3}\rightarrow C^{\otimes3}$$
 by
 $$\Phi_\sigma(x_1\otimes x_2\otimes x_3)=x_{\sigma^{-1}(1)}\otimes x_{\sigma^{-1}(2)}\otimes x_{\sigma^{-1}(3)}.$$
  The following serie of propositions give the connections of Hom-dendriform and Hom-coassociative Rota-Baxter coalgebras to
 Hom-preLie coalgebras.
 \begin{prop}
  Let $(C, \Delta_{-1}, \Delta_1, \alpha)$ be a Hom-dendriform coalgebra. Define the linear map
 $$\Delta=\Delta_1-\tau\circ\Delta_{-1} : C\rightarrow C\otimes C.$$
 Then, $(C, \Delta, \alpha)$ is a Hom-preLie coalgebra.
 \end{prop}
 \begin{proof}
 On the one hand we have
 \begin{eqnarray}
  (\Delta\otimes\alpha)\circ\Delta
  &=&((\Delta_1-\tau\circ\Delta_{-1})\otimes\alpha)\circ\Delta_1
 -(\alpha\otimes(\Delta_1-\tau\circ\Delta_{-1}))\circ\Delta_{-1}\nonumber\\
 &=&(\Delta_1\otimes\alpha)\circ\Delta_1-(\tau\circ\Delta_{-1}\otimes\alpha)\circ\Delta_1
 -\Phi_{(231)}\circ(\alpha\otimes\Delta_1)\circ\Delta_{-1}\nonumber\\
&&+\Phi_{(321)}\circ(\alpha\otimes\Delta_{-1})\circ\Delta_{-1}.\nonumber
 \end{eqnarray}
 On the other hand, 
 \begin{eqnarray}
  (\alpha\otimes\Delta)\circ\Delta
 &=&(\alpha\otimes(\Delta_1-\Delta_{-1}\circ\tau))\circ\Delta_1
 -\Phi_{(312)}\circ((\Delta_1-\Delta_{-1}\circ\tau)\otimes\alpha)\circ\Delta_{-1}\nonumber\\
 &=&(\alpha\otimes(\Delta_1)\circ\Delta_1-(\alpha\otimes(\Delta_{-1}\circ\tau))\circ\Delta_1
 -\Phi_{(312)}\circ(\Delta_1\otimes\alpha)\circ\Delta_{-1}\nonumber\\
&&+\Phi_{(312)}\circ(\Delta_{-1}\circ\tau\otimes\alpha)\circ\Delta_{-1}\nonumber\\
 &=&(\Delta_{-1}\otimes\alpha)\circ\Delta_{1}+(\Delta_{1}\otimes\alpha)\circ\Delta_{1}
 -(\alpha\otimes(\Delta_{-1}\circ\tau))\circ\Delta_1\nonumber\\
 &&-\Phi_{(312)}\circ(\Delta_1\otimes\alpha)\circ\Delta_{-1}+\Phi_{(321)}\circ(\alpha\otimes\Delta_{-1})\circ\Delta_{-1}\nonumber\\
 &&+\Phi_{(321)}\circ(\alpha\otimes\Delta_{1})\circ\Delta_{-1}.\quad(\mbox{by}\;(\ref{c3})\;\mbox{and}\;(\ref{c1}))\nonumber
 \end{eqnarray}
 Thus, using \ref{c2}, one has
 \begin{eqnarray}
 c_\alpha(\Delta)&=&-(\Delta_{-1}\circ\tau\otimes)\circ\Delta_{1}-\Phi_{(321)}\circ(\alpha\otimes\Delta_{1})\circ\Delta_{-1}
 -(\Delta_{-1}\otimes\alpha)\circ\Delta_{1}\nonumber\\
&&+\Phi_{(213)}\circ(\Delta_{1}\otimes\alpha)\circ\Delta_{1}+\Phi_{(312)}\circ(\Delta_{1}\otimes\alpha)\circ\Delta_{-1}\nonumber\\
 &&-\Phi_{(321)}\circ(\alpha\otimes\Delta_{1})\circ\Delta_{-1}\nonumber.
 \end{eqnarray}
 It follows that 
 \begin{eqnarray}
  c_\alpha(\Delta)-\Phi_{(213)}\circ c_\alpha(\Delta)=0,\nonumber
 \end{eqnarray}
 by axioms in Definition \ref{htca}, where  $c_\alpha(\Delta)=(\alpha\otimes\Delta_\cdot)\circ\Delta_\cdot-(\Delta_\cdot\otimes\alpha)\circ\Delta_\cdot.$
 \end{proof}
 The following two propositions are proved as the previous one.
 \begin{prop}
  Let $(C, \Delta_0, \alpha, R)$ be a Hom-coassociative Rota Baxter coalgebra of weight $0$.
Then, $(C, \Delta, \alpha)$ is a Hom-preLie coalgebra, where
 $$\Delta=(R\otimes I)\circ\Delta_0-\tau\circ(I\otimes R)\circ\Delta_0.$$
 \end{prop}
 
 \begin{prop}
   Let $(C, \Delta_0, \alpha, R)$ be a Hom-coassociative Rota Baxter coalgebra of weight $-1$.
 Define the linea map $\Delta : C\rightarrow C\otimes C$ by
 $$\Delta=(R\otimes I)\circ\Delta_0-\tau\circ(I\otimes R)\circ\Delta_0-\Delta_0 .$$
 Then, $(C, \Delta, \alpha)$ is a Hom-preLie coalgebra.
 \end{prop}
 \subsection{Hom-tridendriform comodules}
 \begin{defn}\label{}
 A Hom-tridendriform comodule ove a tridendriform coalgebra $(C, \Delta_{-1}, \Delta_{0}, \Delta_1, \alpha)$ is a Hom-module $(M, \alpha_M)$
  equipped with three linear maps 
 $\Delta_{-1, M}, \Delta_{0, M}, \Delta_{1, M} : C\rightarrow C\otimes M$ satisfing the  following relations :
  \begin{eqnarray}
   (\Delta_{-1}\otimes\alpha_M)\circ\Delta_{-1, M}&=&(\alpha\otimes\Delta_{-1, M}+\alpha\otimes\Delta_{0, M}
 +\alpha\otimes\Delta_{1, M})\circ\Delta_{-1, M},\\
 (\Delta_1\otimes\alpha_M)\circ\Delta_{-1, M}&=&(\alpha\otimes\Delta_{-1, M})\circ\Delta_{1, M},\label{cc2}\\
 \qquad(\alpha\otimes\Delta_{1, M})\circ\Delta_{1, M}&=&(\Delta_{-1}\otimes\alpha_M+\Delta_1\otimes\alpha_M+\Delta_0\otimes\alpha_M)\circ\Delta_{1, M},\\
 (\Delta_{-1}\otimes\alpha_M)\circ\Delta_{0, M}&=&(\alpha\otimes\Delta_{1, M})\circ\Delta_{0, M},\\
 (\Delta_1\otimes\alpha_M)\circ\Delta_{0, M}&=&(\alpha\otimes\Delta_{0, M})\circ\Delta_{1, M},\\
 (\Delta_0\otimes\alpha_M)\circ\Delta_{-1, M}&=&(\alpha\otimes\Delta_{-1, M})\circ\Delta_{0, M},\label{cc6}\\
 (\Delta_0\otimes\alpha_M)\circ\Delta_{0, M}&=&(\alpha\otimes\Delta_{0, M})\circ\Delta_{0, M}.
  \end{eqnarray}
 \end{defn}
 \begin{exam}
  Taking $M=C$ as vector space, $\Delta_{-1, M}=\Delta_{-1}$, $\Delta_{0, M}=\Delta_{0}$, $\Delta_{0, M}=\Delta_{0}$ and 
$\alpha_M=\alpha$, we see that a Hom-tridendriform coalgebra is a comodule over itself.
 \end{exam}

 \begin{thm}
  Let $(C, \Delta_{-1, M}, \Delta_{0, M}, \Delta_{1, M}, \alpha_M)$ be a comodule over the Hom-tridendriform coalgebra
 $(C, \Delta_{-1}, \Delta_0, \Delta_1, \alpha)$.
 Define the linear maps
 $$\Delta_{i, M}^{n, 0}=(\alpha^n\otimes Id_M)\circ\Delta_{i, M} : M\rightarrow C\otimes M, \quad i=-1; 0; 1.$$
 Then, $(M, \Delta_{-1, M}^{n, 0}, \Delta_{0, M}^{n, 0}, \Delta_{1, M}^{n, 0}, \alpha_M)$ is a $C$-comodule.
 \end{thm}
 \begin{proof}
 We only prove (\ref{cc6}). We have
  \begin{eqnarray}
   (\Delta_0\otimes\alpha_M)\circ\Delta_{-1, M}^{n, 0}
 &=&(\Delta_0\otimes\alpha_M)\circ(\alpha^n\otimes Id_M)\circ\Delta_{-1, M}\nonumber\\
 &=&(\Delta_0\circ\alpha^n\otimes\alpha_M)\circ\Delta_{-1, M}\nonumber\\
 &=&(\alpha^n\otimes\alpha^n\otimes Id_M)\circ(\Delta_0\otimes\alpha_M)\circ\Delta_{-1, M}\nonumber\\
 &=&(\alpha^n\otimes\alpha^n\otimes Id_M)\circ(\alpha\otimes\Delta_{-1, M})\circ\Delta_{0, M}\nonumber\\
 &=&(\alpha^{n+1}\otimes(\alpha^n\otimes Id_M)\circ\Delta_{-1, M})\circ\Delta_{0, M}\nonumber\\
 &=&(\alpha\otimes\Delta_{-1, M}^{n, 0})\circ\Delta_{0, M}^{n, 0}.\nonumber
 \end{eqnarray}
 The proofs of the other axioms are similar.
 \end{proof}
 
 \begin{rem}
 We can prove by a straightforward calculation that if $(M, \Delta_{-1, M}, \Delta_{0, M}$, $ \Delta_{1, M}, \alpha_M)$ is a $C$-comodule 
 with $(C, \Delta_{-1}, \Delta_0, \Delta_1, \alpha)$ a Hom-tridendriform coalgebra,
 then $(M, \Delta_{-1, M}^{0, k}=\Delta_{-1, M}\circ\alpha_M^{2^k-1}, \Delta_{0, M}^{0, k}=\Delta_{0, M}\circ\alpha_M^{2^k-1}, 
 \Delta_{1, M}^{0, k}=\Delta_{1, M}\circ\alpha_M^{2^k-1}, \alpha_M^{2^k-1})$ is a comodule over 
 $(C, \Delta_{-1}^{0, k}=\Delta_{-1}\circ\alpha^{2^k-1}, \Delta_{0, M}^{0, k}=\Delta_0\circ\alpha^{2^k-1}, 
 \Delta_{1, M}^{0, k}=\Delta_1\circ\alpha^{2^k-1}, \alpha^{2^k-1})$.
 (resp. $\Delta_{i, M}^{n, k}=(\alpha^n\otimes Id_M)\circ\Delta_{i, M}\circ\alpha_M^{2^k}$)
 confere to $s$ (resp. $m$) a structure  of Hom-tridendriform comodule.
 \end{rem}
 
\section{Post-Hom-Lie coalgebra}
In this section, we define post-Hom-Lie coalgebras and give some properties in relation with other algebraic structures.
\begin{defn}\label{tdd}
 A post-Hom-Lie coalgebra is a vector space $C$ endowed with two linear comultiplications $\Delta : C\rightarrow C\otimes C$ and $\gamma : C\rightarrow C\otimes C$ and, $\alpha : C\rightarrow C$ a linear map such that:
\begin{eqnarray}
 \gamma=-\tau\circ\gamma;\ \ \gamma\circ\alpha= \alpha^{\otimes^{2}}\circ\gamma\hspace{7cm} \label{dplc1}
 \\
(1+\xi+\xi^2)\circ(\alpha\otimes\gamma)\circ\gamma=0,\hspace{6cm}\label{dplc2}\\
(\alpha\otimes\gamma)\circ\Delta-(\Delta\otimes \alpha)\circ\gamma-(\tau\otimes I)\circ(\alpha\otimes\Delta)\circ\gamma=0,\hspace{4cm}\label{dplc3}\\
(\Delta\otimes \alpha)\circ\Delta\!-\!(\alpha\otimes\Delta)\circ\Delta\!-\!(\tau\otimes I)\circ(\Delta\otimes \alpha)\circ\Delta\!
+\!(\tau\otimes I)\circ(\alpha\otimes\Delta)\circ\Delta\!+\!(\gamma\otimes \alpha)\circ\Delta=0,\label{dplc4}
\end{eqnarray}
where $\tau(x\otimes y)=y\otimes x$, $\xi(x\otimes y\otimes z)=y\otimes z\otimes x$ and $\xi^2(x\otimes y\otimes z)=z\otimes x\otimes y$ for all $x, y,z\in C$.
\end{defn}

\begin{rem}
We have the following observations :
 $$\xi=(I\otimes \tau)\circ(\tau\otimes I)\quad\mbox{and}\quad \xi^2=(\tau\otimes I)\circ(I\otimes\tau).$$
\end{rem}

\begin{rem}
 The qualities (\ref{dplc1})-(\ref{dplc4}) may be rewritten respectively as :
\begin{eqnarray}
 x_1\otimes x_2=-x_2\otimes x_1;\ \alpha(x)_1\otimes \alpha(x)_2=\alpha(x_1)\otimes \alpha(x_2)\hspace{3cm}\label{dplc1a}
 \\
\alpha(x_1)\otimes x_{21}\otimes x_{22}+x_{21}\otimes x_{22}\otimes \alpha(x_1)+x_{22}\otimes \alpha(x_1)\otimes x_{21}=0,\hspace{3cm}\label{dplc2a}
\\
\alpha(x_{(1)})\otimes x_{(2)1}\otimes x_{(2)2}-x_{1(1)}\otimes x_{1(2)}\otimes \alpha(x_2)-x_{2(1)}\otimes \alpha(x_1)\otimes x_{2(2)}=0,\hspace{2cm}\label{dplc3a}
\\
x_{(1)(1)}\otimes x_{(1)(2)}\otimes\alpha(x_{(2)})-\alpha(x_{(1)})\otimes x_{(2)(1)}\otimes x_{(2)(2)}
-x_{(1)(2)}\otimes x_{(1)(1)}\otimes\alpha( x_{(2)})\label{dplc4a}\hspace{1,5cm}
\\
+x_{(2)(1)}\otimes \alpha(x_{(1)})\otimes x_{(2)(2)}+x_{(1)1}\otimes x_{(1)2}\otimes \alpha(x_{(2)})=0\nonumber.
\end{eqnarray}
\end{rem}

\begin{rem}
 The axioms (\ref{dplc3a}) and (\ref{dplc4a}) are respectively equivalent to :
\begin{eqnarray}
 x_{(2)1}\otimes \alpha(x_{(1)})\otimes x_{(2)2}- x_{1(2)}\otimes x_{1(1)}\otimes \alpha(x_2)- \alpha(x_1)\otimes x_{2(1)}\otimes x_{2(2)}=0,\hspace{2cm}\label{dplc3b}
 \\
 x_{(1)(2)}\otimes x_{(1)(1)}\otimes \alpha(x_{(2)})- x_{(2)(1)}\otimes \alpha(x_{(1)})\otimes x_{(2)(2)}
- x_{(1)(1)}\otimes x_{(1)(2)}\otimes \alpha(x_{(2)})\label{dplc4b}\hspace{1,5cm}\\
+ \alpha(x_{(1)})\otimes x_{(2)(1)}\otimes x_{(2)(2)}+ x_{(1)2}\otimes x_{(1)1}\otimes \alpha(x_{(2)})=0\nonumber.
\end{eqnarray}
\end{rem}

\begin{rem}
In any Hom-Lie coalgebra we have :
\begin{eqnarray}
x_{11}\otimes x_{12}\otimes \alpha(x_2)&=&(\gamma\otimes \alpha)\gamma(x)=-(\gamma\otimes \alpha)\tau\gamma(x)=-x_{21}\otimes x_{22}\otimes\alpha(x_1). \label{re1}\\
x_{21}\otimes \alpha(x_1)\otimes x_{22}&=&(\tau\otimes I)(\alpha\otimes \gamma)\gamma(x)
=-(\tau\otimes I)(\alpha\otimes \tau\gamma)\gamma(x)=-x_{22}\otimes \alpha(x_1)\otimes x_{21}\label{re2}.
 \end{eqnarray}
 \begin{eqnarray}
  x_{12}\otimes x_{11}\otimes \alpha(x_2)&=&(\tau\otimes I)(x_{11}\otimes x_{12}\otimes \alpha(x_2))=(\tau\otimes I)(\gamma\otimes \alpha)\gamma(x)\nonumber\\
&=&(\gamma\otimes \alpha)\tau\gamma(x)=x_{21}\otimes x_{22}\otimes \alpha(x_1)\\
&=&-(\gamma\otimes \alpha)\gamma(x)=-x_{11}\otimes x_{12}\otimes \alpha(x_2).
 \end{eqnarray}
\end{rem}
\begin{rem}
In any post-Hom-Lie coalgebra we have :
\begin{eqnarray}
x_{2(1)}\otimes x_{2(2)}\otimes \alpha(x_{1})&=&(\Delta\otimes \alpha)\tau\gamma(x)=-(\Delta\otimes \alpha)\gamma(x)=x_{1(1)}\otimes x_{1(2)}\otimes \alpha(x_{2}).
\end{eqnarray}
\end{rem}
\begin{defn}
 Let $(A, \gamma,  \Delta, \alpha)$ and $(A', \gamma', \Delta', \alpha')$ be two post-Hom-Lie coalgebras.
 A morphism of post-Hom-Lie coalgebras is a linear map $f : A\rightarrow A'$ such that
$$\Delta\circ f=(f\otimes f)\circ\Delta',\quad \gamma\circ f=(f\otimes f)\circ\gamma' \quad\mbox{and}\quad \alpha'\circ f=f\circ\alpha.$$
\end{defn}
\begin{lem}\label{hplm1}
  Let $(L, \gamma, \Delta, \alpha)$ be a post-Hom-Lie coalgebra and $\beta : L\rightarrow L$ be a
post-Hom-Lie coalgebra  endomorphism. Then
$L_\beta=(L, \gamma_\beta=\gamma\circ\beta, \Delta_\beta=\Delta\circ\beta, \beta\circ\alpha)$ is a post-Hom-Lie coalgebra.

Moreover, suppose that $(L', \gamma', \Delta')$ is a post-Lie coalgebra and $\alpha' : L'\rightarrow L'$ is a
post-Lie coalgebra  endomorphism. If $f : L\rightarrow L'$ is a post-Lie coalgebra morphism that satisfies
 $\beta\circ f= f\circ\alpha'$, then $ f : L_\beta\rightarrow L'_{\alpha'}$
is a morphism of post-Hom-Lie coalgebras.
\end{lem}
\begin{proof}
 It is clear that $(L,\gamma_\beta, \beta\circ\alpha)$ is a Hom-Lie coalgebra. Then,
 we have :
\begin{eqnarray}
 &&(\beta\alpha\otimes\gamma_\beta)\circ\Delta_\beta-\circ(\Delta_\beta\otimes\beta\alpha)\circ\gamma_\beta -(I\otimes \tau)\circ(\beta\alpha\otimes\Delta_\beta)\circ\gamma_\beta\nonumber\\
&=&(\beta\alpha\otimes\gamma\circ\beta)\circ\Delta\circ\beta-(\Delta\circ\beta\otimes\beta\alpha)\circ\gamma\circ\beta -(I\otimes \tau)\circ(\beta\alpha\otimes\Delta\circ\beta)\circ\gamma\circ\beta\nonumber\\
&=&(\alpha\otimes\gamma)\circ\Delta-(\Delta\otimes\alpha)\circ\gamma -(I\otimes \tau)\circ(\alpha\otimes\Delta)\circ\gamma)\circ\beta^2=0\nonumber.
\end{eqnarray}
Axiom (\ref{dplc4}) has a similar proof.

For the second part, observe that
\begin{eqnarray}
 \gamma_\beta\circ f\!\!\!&=&\!\!\!\!\gamma\circ\beta\circ f=\gamma\circ f\circ\alpha'\nonumber\\
&=&(f\otimes f)\circ \gamma'\circ\alpha'
=(f\otimes f)\circ\gamma'_{\alpha'},\label{ccc}\nonumber
\end{eqnarray}
 and
\begin{eqnarray}
\Delta_\beta\circ f\!\!\!\!&=&\!\!\!\!\Delta\circ\beta\circ f=\Delta\circ f\circ\alpha'
=(f\otimes f)\circ \Delta'\circ\alpha'=(f\otimes f)\circ\Delta'_{\alpha'}\nonumber.
\end{eqnarray}
Hence, the conclusion holds.
\end{proof}
\begin{prop}
 Let $(C, \Delta, \gamma, \alpha)$ be a post-Hom-Lie coalgebra. Let us define
$$\tilde\Delta=\Delta+\gamma\quad\mbox{and}\quad\tilde\gamma=-\gamma.$$
Then, $(C, \tilde\Delta, \tilde\gamma, \alpha)$ is a post-Hom-Lie coalgebra.
\end{prop}
\begin{proof}
 The skew-cosymmetry and the Hom-coJacobi identity are trivial. Let us prove the last two axioms. For any $x\in C$, we have on the one hand
\begin{eqnarray}
 &&\qquad(\alpha\otimes\tilde\gamma)\tilde\Delta(x)-(\tilde\Delta\otimes \alpha)\tilde\gamma(x)-(\tau\otimes\tilde\Delta)\tilde\gamma(x)=\nonumber\\
&&=-(\alpha\otimes\gamma)(\Delta+\gamma)(x)+((\Delta+\gamma)\otimes \alpha)\gamma(x)+(\tau\otimes I)(\alpha\otimes(\Delta+\gamma))\gamma(x)\nonumber\\
&&=-(\alpha\otimes\gamma)\Delta(x)-(\alpha\otimes\gamma)\gamma(x)+(\Delta\otimes \alpha)\gamma(x)+(\gamma\otimes \alpha)\gamma(x)\nonumber\\
&&\qquad+(\tau\otimes I)(\alpha\otimes\Delta)\gamma(x)+(\tau\otimes I)(\alpha\otimes\gamma)\gamma(x)\nonumber\\
&&=\Big(-(\alpha\otimes\gamma)\Delta(x)+(\Delta\otimes \alpha)\gamma(x)+(\tau\otimes I)(\alpha\otimes\Delta)\gamma(x)\Big)\nonumber\\
&&\qquad+\Big(-(\alpha\otimes\gamma)\gamma(x)+(\gamma\otimes \alpha)\gamma(x)+(\tau\otimes I)(\alpha\otimes\gamma)\gamma(x)\Big).\nonumber
\end{eqnarray}
The firt big parenthesis cancel by (\ref{dplc3}). For the second, we have 
\begin{eqnarray}
 &&-(\alpha\otimes\gamma)\gamma(x)+(\gamma\otimes \alpha)\gamma(x)+(\tau\otimes I)(\alpha\otimes\gamma)\gamma(x)=\nonumber\\
&&\qquad=-\alpha(x_1)\otimes x_{21}\otimes x_{22}+x_{11}\otimes x_{12}\otimes \alpha(x_2)+x_{21}\otimes \alpha(x_1)\otimes x_{22}.\nonumber
\end{eqnarray}
Which cancels by Hom-coJacobi identity, using (\ref{re1}) and (\ref{re2}).\\
On the other hand, 
\begin{eqnarray}
 &&(\tilde\Delta\otimes \alpha)\tilde\Delta(x)-(\alpha\otimes \tilde\Delta)\tilde\Delta(x) -(\tau\otimes I)(\tilde\Delta\otimes \alpha)\tilde\Delta(x) 
+(\tau\otimes I)(\alpha\otimes \tilde\Delta)\tilde\Delta(x) +(\tilde\gamma\otimes \alpha)\tilde\Delta(x)=\nonumber\\
&&\quad=\tilde\Delta(x_{(1)})\otimes \alpha(x_{(2)})+\tilde\Delta(x_{1})\otimes \alpha(x_{2})
-\alpha(x_{(1)})\otimes\tilde\Delta(x_{(2)})-\alpha(x_{1})\otimes\tilde\Delta(x_{2})\nonumber\\
&&\qquad-(\tau\otimes I)(\tilde\Delta(x_{(1)})\otimes \alpha(x_{(2)}))
- (\tau\otimes I)(\tilde\Delta(x_{1})\otimes \alpha(x_{2}))
+(\tau\otimes I)(\alpha(x_{(1)})\otimes\tilde\Delta(x_{(2)}))\nonumber\\
&&\qquad+(\tau\otimes I)(\alpha(x_{1})\otimes\tilde\Delta(x_{2}))
+\tilde\gamma(x_{(1)})\otimes \alpha(x_{(2)})+\tilde\gamma(x_1)\otimes \alpha(x_2)\nonumber\\
&&\quad=x_{(1)(1)}\otimes x_{(1)(2)}\otimes \alpha(x_{(2)})+x_{(1)1}\otimes x_{(1)2}\otimes \alpha(x_{(2)})+x_{1(1)}\otimes x_{1(2)}\otimes \alpha(x_{2})
+x_{11}\otimes x_{12}\otimes \alpha(x_{2})\nonumber\\
&&\qquad-\alpha(x_{(1)})\otimes x_{(2)(1)}\otimes x_{(2)(2)}-\alpha(x_{(1)})\otimes x_{(2)1}\otimes x_{(2)2}
-\alpha(x_{1})\otimes x_{2(1)}\otimes x_{2(2)}-\alpha(x_{1})\otimes x_{21}\otimes x_{22}\nonumber\\
&&\qquad-x_{(1)(2)}\otimes x_{(1)(1)}\otimes \alpha(x_{(2)})
-x_{(1)2}\otimes x_{(1)1}\otimes \alpha(x_{(2)})-x_{1(2)}\otimes x_{1(1)}\otimes \alpha(x_{2})-x_{12}\otimes x_{11}\otimes \alpha(x_{2})\nonumber\\
&&\qquad+x_{(2)(1)}\otimes \alpha(x_{(1)})\otimes x_{(2)(2)}+x_{(2)1}\otimes \alpha(x_{(1)})\otimes x_{(2)2}+x_{2(1)}\otimes\alpha(x_{1})\otimes x_{2(2)}
+x_{21}\otimes \alpha(x_{1})\otimes x_{22}\nonumber\\
&&\qquad-x_{(1)1}\otimes x_{(1)2}\otimes\alpha(x_{(2)})-x_{11}\otimes x_{12}\otimes \alpha(x_{2})\nonumber.
\end{eqnarray}
By ordering and grouping, it comes
\begin{eqnarray}
&&(\tilde\Delta\otimes \alpha)\tilde\Delta(x)\!-\!(\alpha\otimes \tilde\Delta)\tilde\Delta(x)\! -\!(\tau\otimes I)(\tilde\Delta\otimes \alpha)\tilde\Delta(x) 
\!+\!(\tau\otimes I)(\alpha\otimes \tilde\Delta)\tilde\Delta(x)\! +\!(\tilde\gamma\otimes \alpha)\tilde\Delta(x)=\nonumber\\
&&\quad=\Big(x_{(1)(1)}\otimes x_{(1)(2)}\otimes \alpha(x_{(2)})-\alpha(x_{(1)})\otimes x_{(2)(1)}\otimes x_{(2)(2)}
-x_{(1)(2)}\otimes x_{(1)(1)}\otimes \alpha(x_{2})\nonumber\\
&&\qquad+x_{(2)(1)}\otimes \alpha(x_{(1)})\otimes x_{(2)(2)}-x_{(1)1}\otimes x_{(1)2}\otimes \alpha(x_{(2)})\Big)
+\Big(-\alpha(x_{1})\otimes x_{21}\otimes x_{22}-x_{12}\otimes x_{11}\otimes \alpha(x_{2})\nonumber\\
&&\qquad+x_{21}\otimes \alpha(x_{1})\otimes x_{22} \Big)+\Big(x_{11}\otimes x_{12}\otimes \alpha(x_{2})-\alpha(x_{1})\otimes x_{21}\otimes x_{22}\Big)
+\Big(x_{1(1)}\otimes x_{1(2)}\otimes \alpha(x_{2})\nonumber\\
&&\qquad-\alpha(x_{(1)})\otimes x_{(2)1}\otimes x_{(2)2}+x_{2(1)}\otimes\alpha(x_{1})\otimes x_{2(2)}\Big)
+\Big(x_{(2)1}\otimes\alpha(x_{(1)})\otimes x_{(2)2}-\alpha(x_{1})\otimes x_{2(1)}\otimes x_{2(2)}\nonumber\\
&&\qquad-x_{1(2)}\otimes x_{1(1)}\otimes \alpha(x_{2}) \Big)
+\Big(x_{(1)1}\otimes x_{(1)2}\otimes\alpha(x_{2})-x_{(1)2}\otimes x_{(1)1}\otimes \alpha(x_{(2)})\Big).\nonumber
\end{eqnarray}
\end{proof}

\begin{prop}
 Let $(C, \Delta, \gamma)$ be a post-Hom-Lie coalgebra. Then, $C$ is an admissible Hom-Lie coalgebra with the map
$$\tilde\Delta(x)=\Delta(x)+\frac{1}{2}\gamma(x).$$
\end{prop}
\begin{proof}
 For any $x\in C$, we have
\begin{eqnarray}
 &&\qquad(\alpha\otimes\tilde\gamma)\tilde\gamma(x)=\nonumber\\
&=&(\alpha\otimes(\Delta+\frac{1}{2}\gamma-\tau\Delta-\frac{1}{2}\tau\gamma))\tilde\Delta(x)\nonumber\\
&&-(\tau\otimes I)(I\otimes\tau)((\alpha\otimes(\Delta+\frac{1}{2}\gamma-\tau\Delta-\frac{1}{2}\tau\gamma)\otimes \alpha)\tilde\Delta(x)\nonumber\\
&=&(\alpha\otimes(\Delta-\tau\Delta+\gamma))\Delta(x)+\frac{1}{2}(\alpha\otimes(\Delta-\tau\Delta+\gamma))\gamma(x)\nonumber\\
&&-(\tau\otimes I)(I\otimes\tau)((\Delta-\tau\Delta+\gamma)\otimes \alpha)\Delta(x)
-\frac{1}{2}(\tau\otimes I)(I\otimes\tau)((\Delta-\tau\Delta+\gamma)\otimes \alpha)\gamma(x)\nonumber
\end{eqnarray}

\begin{eqnarray}
&=&(\alpha\otimes\Delta)\Delta(x)-(I\otimes\tau)(\alpha\otimes\Delta)\Delta(x)+(\alpha\otimes\gamma)\Delta(x)
-(\tau\otimes I)(I\otimes\tau)(\Delta\otimes \alpha)\Delta(x)\nonumber\\
&&+(\tau\otimes I)(I\otimes\tau)(\tau\otimes I)(\Delta\otimes \alpha)\Delta(x)
-(\tau\otimes I)(I\otimes\tau)(\gamma\otimes \alpha)\Delta(x)+\frac{1}{2}(\alpha\otimes\Delta)\gamma(x)\nonumber\\
&&-\frac{1}{2}(\tau\otimes I)(\alpha\otimes\Delta)\gamma(x)+\frac{1}{2}(\alpha\otimes\gamma)\gamma(x)
-\frac{1}{2}(\tau\otimes I)(I\otimes\tau)(\Delta\otimes \alpha)\gamma(x)\nonumber\\
&&-\frac{1}{2}(\tau\otimes I)(I\otimes\tau)(\gamma\otimes \alpha)\gamma(x)
+\frac{1}{2}(\tau\otimes I)(I\otimes\tau)(\tau\otimes I)(\Delta\otimes \alpha)\gamma(x)\nonumber\\
&=&(\alpha\otimes\Delta)\Delta(x)-(I\otimes\tau)(\alpha\otimes\Delta)\Delta(x)+(\alpha\otimes\gamma)\Delta(x)
-(\tau\otimes I)(I\otimes\tau)(\Delta\otimes \alpha)\Delta(x)\nonumber\\
&&+(\tau\otimes I)(I\otimes\tau)(\tau\otimes I)(\Delta\otimes \alpha)\Delta(x)
-(\tau\otimes I)(I\otimes\tau)(\gamma\otimes \alpha)\Delta(x)\nonumber\\
&&-(I\otimes\tau)(\alpha\otimes\Delta)\gamma(x)+(I\otimes\tau)(\alpha\otimes\Delta)\gamma(x)
+(\alpha\otimes\gamma)\gamma(x)\nonumber\\
&=&\alpha(x_{(1)})\otimes x_{(2)(1)}\otimes x_{(2)(2)}-\alpha(x_{(1)})\otimes x_{(2)(2)}\otimes x_{(2)(1)}+\alpha(x_{(1)})\otimes x_{(2)1}\otimes x_{(2)2}\nonumber\\
&&-\alpha(x_{(2)})\otimes x_{(1)(1)}\otimes x_{(1)(2)}+\alpha(x_{(2)})\otimes x_{(1)(2)}\otimes x_{(1)(1)}-\alpha(x_{(2)})\otimes x_{(1)1}\otimes x_{(1)2}\nonumber\\
&&-\alpha(x_{1})\otimes x_{2(2)}\otimes x_{2(1)}+(\alpha\otimes\gamma)\gamma.\nonumber
\end{eqnarray}
It follows, by grouping the terms,
\begin{eqnarray}
 &&\qquad(\alpha\otimes\tilde\gamma)\tilde\gamma(x)+\xi(\alpha\otimes\tilde\gamma)\tilde\gamma(x)+\xi^2(\alpha\otimes\tilde\gamma)\tilde\gamma(x)=\nonumber\\
&=&\Big(\alpha(x_{(1))}\otimes x_{(2)(1)}\otimes x_{(2)(2)}-x_{(1)(1)}\otimes x_{(1)(2)}\otimes \alpha(x_{(2)})+x_{(1)(2)}\otimes x_{(1)(1)}\otimes \alpha(x_{(2)})\nonumber\\
&&-x_{(1)1}\otimes x_{(1)2}\otimes \alpha(x_{(2)})-x_{(2)(1)}\otimes \alpha(x_{(1)})\otimes x_{(2)(2)}\Big)
+\Big(-\alpha(x_{(1)})\otimes x_{(2)(2)}\otimes x_{(2)(1)}\nonumber\\
&&+x_{(2)(1)}\otimes x_{(2)(2)}\otimes \alpha(x_{(1)})-x_{(1)(2)}\otimes \alpha(x_{(2)})\otimes x_{(1)(1)}
+x_{(1)(1)}\otimes \alpha(x_{(2)})\otimes x_{(1)(2)}\nonumber\\
&&-x_{(1)2}\otimes \alpha(x_{(2)})\otimes x_{(1)1}\Big)
+\Big(\alpha(x_{(1)})\otimes x_{(2)1}\otimes x_{(2)2}+x_{2(1)}\otimes x_{2(2)}\otimes \alpha(x_{1})-x_{2(1)}\otimes \alpha(x_{1})\otimes x_{2(2)}\Big)\nonumber\\
&&+\Big(-\alpha(x_{(2)})\otimes x_{(1)(1)}\otimes x_{(1)(2)}+\alpha(x_{(2)})\otimes x_{(1)(2)}\otimes x_{(1)(1)}-\alpha(x_{(2)})\otimes x_{(1)1}\otimes x_{(1)2}\nonumber
\end{eqnarray}
\begin{eqnarray}
&&-x_{(2)(2)}\otimes x_{(2)(1)}\otimes \alpha(x_{(1)})+x_{(2)(2)}\otimes \alpha(x_{(1)})\otimes x_{(2)(1)}\Big)
+\Big(-\alpha(x_{1})\otimes x_{2(2)}\otimes x_{2(1)}\nonumber\\
&&+\alpha(x_{1})\otimes x_{2(1)}\otimes x_{2(2)}+x_{(2)1}\otimes x_{(2)2}\otimes \alpha(x_{(1)})\Big)
+\Big(-x_{2(2)}\otimes x_{2(1)}\otimes \alpha(x_{1})\nonumber\\
&&+\alpha(x_{(1)})\otimes x_{(2)2}\otimes x_{(2)1}+x_{2(2)}\otimes \alpha(x_{1})\otimes x_{2(1)}\Big)\nonumber.
\end{eqnarray}
The second hand side cancels by (\ref{dplc3a}) and (\ref{dplc4a}).
\end{proof}
\begin{lem}\label{le1}
 Let $(C, \Delta, \gamma, \alpha)$ be a post-Hom-Lie coalgebra. Let us define the map $\tilde\Delta : C\rightarrow C\otimes C$ by
$$\tilde\Delta(x)=\Delta(x)+\frac{1}{2}\gamma(x)$$
for any $x\in C$. Then,
$$as_{\tilde\Delta}(x)-(\tau\otimes I)as_{\tilde\Delta}(x)=-\frac{1}{2}(\tau\otimes I)(\gamma\otimes \alpha)\gamma(x).$$
\end{lem}
\begin{proof}
We have,
\begin{eqnarray}
 as_{\tilde\Delta}(x)
&=&(\tilde\Delta\otimes \alpha)\tilde\Delta(x)-(\alpha\otimes\tilde\Delta)\tilde\Delta(x)\nonumber\\
&=&((\Delta+\frac{1}{2}\gamma)\otimes \alpha)(\Delta+\frac{1}{2}\gamma)(x)-(\alpha\otimes(\Delta+\frac{1}{2}\gamma))(\Delta+\frac{1}{2}\gamma)(x)\nonumber\\
&=&(\Delta\otimes \alpha)\Delta(x)+\frac{1}{2}(\gamma\otimes \alpha)\Delta(x)+
\frac{1}{2}(\Delta\otimes \alpha)\gamma(x)+\frac{1}{4}(\gamma\otimes \alpha)\gamma(x)\nonumber\\
&&-(\alpha\otimes\Delta)\Delta(x)-\frac{1}{2}(\alpha\otimes\gamma)\Delta(x)
-\frac{1}{2}(\alpha\otimes\Delta)\gamma(x)-\frac{1}{4}(\alpha\otimes\gamma)\gamma(x).\nonumber
\end{eqnarray}
It comes, by grouping the terms,
\begin{eqnarray}
&&\qquad as_{\tilde\Delta}(x)-(\tau\otimes \alpha)as_{\tilde\Delta}(x)=\nonumber\\
&&=\Big((\Delta\otimes \alpha)\Delta(x)-(\alpha\otimes\Delta)\Delta(x)-(\tau\otimes I)(\Delta\otimes \alpha)\Delta(x)
+(\tau\otimes I)(\alpha\otimes\Delta)\Delta(x)\Big)\nonumber\\
&&\quad+\Big(\frac{1}{2}(\gamma\otimes \alpha)\Delta(x)-\frac{1}{2}(\tau\otimes I)(\gamma\otimes \alpha)\Delta(x) \Big)
+\Big(\frac{1}{2}(\Delta\otimes \alpha)\gamma(x)-\frac{1}{2}(\alpha\otimes\gamma)\Delta(x)\nonumber\\
&&\quad+\frac{1}{2}(\tau\otimes I)(\alpha\otimes\Delta)\gamma(x)\Big)
+\Big(\frac{1}{4}(\gamma\otimes \alpha)\gamma(x)-\frac{1}{4}(\alpha\otimes\gamma)\gamma(x)+\frac{1}{4}(\tau\otimes I)(\alpha\otimes\gamma)\gamma(x)\Big)\nonumber\\
&&+\Big(-\frac{1}{2}(\tau\otimes I)(\Delta\otimes \alpha)\gamma(x)+\frac{1}{2}(\tau\otimes I)(\alpha\otimes\gamma)\Delta(x)
-\frac{1}{2}(\alpha\otimes\Delta)\gamma(x)\Big)-\frac{1}{4}(\alpha\otimes\gamma)\gamma(x)\nonumber\\
&&=-(\gamma\otimes \alpha)\Delta(x)+(\gamma\otimes \alpha)\Delta(x)
+\Big(\frac{1}{2}(\Delta\otimes \alpha)\gamma(x)-\frac{1}{2}(\alpha\otimes\gamma)\Delta(x)\nonumber\\
&&\quad+\frac{1}{2}(\tau\otimes I)(\alpha\otimes\Delta)\gamma(x)\Big)
+\Big(\frac{1}{4}(\gamma\otimes \alpha)\gamma(x)-\frac{1}{4}(\alpha\otimes\gamma)\gamma(x)+\frac{1}{4}(\tau\otimes I)(\alpha\otimes\gamma)\gamma(x)\Big)\nonumber\\
&&+\Big(-\frac{1}{2}(\tau\otimes I)(\Delta\otimes \alpha)\gamma(x)+\frac{1}{2}(\tau\otimes I)(\alpha\otimes\gamma)\Delta(x)
-\frac{1}{2}(\alpha\otimes\Delta)\gamma(x)\Big)-\frac{1}{4}(\alpha\otimes\gamma)\gamma(x)\nonumber\\
&&=-\frac{1}{4}(\alpha\otimes\gamma)\gamma(x).\nonumber
\end{eqnarray}
The three other bracket cancel by (\ref{dplc2}) and (\ref{dplc3}).
\end{proof}
\begin{prop}
 Let $(C, \Delta, \gamma, \alpha)$ be a post-Hom-Lie coalgebra. Let us define the map $\tilde\Delta : C\rightarrow C\otimes C$ by
$$\tilde\Delta(x)=\Delta(x)+\frac{1}{2}\gamma(x)$$
for any $x\in C$. Then, $(C, \tilde\Delta, \alpha)$ is an admissible Hom-Lie coalgebra if and only if
$$(I+\xi+\xi^2)\Big(as_{\tilde\Delta}(x)-(\tau\otimes I)as_{\tilde\Delta}(x)\Big)=0.$$
\end{prop}
\begin{proof}
 It comes from Lemma \ref{le1}.
\end{proof}

\begin{defn}
 Let $(C, \gamma, \alpha)$ be a Hom-Lie coalgebra. A linear map $R : C\rightarrow C$ is called a Rota-Baxter operator on $C$ of weight $\lambda\in\mathbb{K}$
if 
\begin{eqnarray}
    R\circ\alpha=\alpha\circ R,
\end{eqnarray}
\begin{eqnarray}
 (R\otimes R)\gamma(x)=((R\otimes I)\gamma+(I\otimes R)\gamma+\lambda\gamma)R(x)
\end{eqnarray}
for any $x\in C$.
\end{defn}
In this case, we say that $(C, \gamma, R, \alpha)$ is a Rota-Baxter Hom-Lie coalgebra.

\begin{prop}
 Let $(C, \gamma, R, \alpha)$ be a Rota-Baxter Hom-Lie coalgebra. Let us define
$$\tilde\gamma=\lambda\gamma\quad\mbox{and}\quad\tilde\Delta=(R\otimes I)\circ\gamma.$$
Then, $(C, \tilde\Delta, \tilde\gamma, \alpha)$ is a post-Hom-Lie coalgebra.
\end{prop}
\begin{proof}
The first two axioms are trivial. Next, 
\begin{eqnarray}
&&\qquad(\tilde\Delta\otimes \alpha)\tilde\Delta(x)-(\alpha\otimes \tilde\Delta)\tilde\Delta(x)-(\tau\otimes I)(\tilde\Delta\otimes \alpha)\tilde\Delta(x)\nonumber\\ 
&&\hspace{6cm}+(\tau\otimes I)(\alpha\otimes \tilde\Delta)\tilde\Delta(x)\! +\!(\tilde\gamma\otimes \alpha)\tilde\Delta(x)=\nonumber\\
&&=((R\otimes I)\gamma\otimes \alpha)(R\otimes I)\gamma(x)-(\alpha\otimes(R\otimes I)\gamma)(R\otimes I)\gamma(x)\nonumber\\
&&-(\tau\otimes I)((R\otimes I)\gamma\otimes \alpha)(R\otimes I)\gamma(x)\nonumber\\
&&\quad+(\tau\otimes I)(\alpha\otimes(R\otimes I)\gamma)(R\otimes I)\gamma(x)+\lambda(\gamma\otimes \alpha)(R\otimes I)\gamma(x)\nonumber\\
&&=((R\otimes I)\gamma\otimes \alpha)(R\otimes I)\gamma(x)
+\Big(-(\alpha\otimes(R\otimes I)\gamma)(R\otimes I)\gamma(x)+(\tau\otimes I)(\alpha\otimes \tilde\Delta)\Delta(x)\Big)\nonumber\\
&&\quad -\!(\tau\otimes I)(\Delta\otimes \alpha)\Delta(x)+\lambda(\gamma\otimes \alpha)(R\otimes I)\gamma(x)\nonumber\\
&&=(R\otimes I\otimes I)(\gamma\otimes \alpha)(R\otimes I)\gamma(x)-(R\otimes R\otimes I)(\gamma\otimes \alpha)\gamma(x)\nonumber\\
&&\quad+(\tau\otimes I)((R\otimes I)\gamma\otimes \alpha)(R\otimes I)\gamma(x)+\lambda(\gamma R\otimes I)\gamma(x)\nonumber\\
&&=\Big((R\otimes I\otimes I)(\gamma R\otimes \alpha)-(R\otimes R\otimes I)(\gamma\otimes \alpha)+((I\otimes R)\gamma R\otimes \alpha)\gamma
+\lambda(\gamma R\otimes \alpha)\Big)\gamma(x)\nonumber\\
&&=\Big((R\otimes I)\gamma R\otimes \alpha-(R\otimes R)\gamma\otimes \alpha+(I\otimes R)\gamma R\otimes \alpha+\lambda\gamma R\otimes \alpha\Big)\gamma(x)\nonumber\\
&&=\Big\{\Big((R\otimes I)\gamma R-(R\otimes R)\gamma+(I\otimes R)\gamma R+\lambda\gamma R\Big)\otimes \alpha\Big\}\gamma(x)\nonumber\\
&&=0.\nonumber
\end{eqnarray}
Finally,
\begin{eqnarray}
&& \qquad(\alpha\otimes\tilde\gamma)\tilde\Delta-(\tilde\Delta\otimes \alpha)\tilde\gamma-(\tau\otimes I)(\alpha\otimes\tilde\Delta)\tilde\gamma=\nonumber\\
&&=(\alpha\otimes\lambda\gamma)(R\otimes I)\gamma-((R\otimes I)\gamma\otimes \alpha)\lambda\gamma
-(\tau\otimes I)(\alpha\otimes(R\otimes I)\gamma)\lambda\gamma\nonumber\\
&&=\lambda(R\otimes\gamma)\gamma-\lambda(R\otimes I\otimes \alpha)(\gamma\otimes I)\gamma
-\lambda(\tau\otimes I)(\alpha\otimes R\otimes I)(I\otimes\gamma)\gamma\nonumber\\
&&=\lambda\Big((R\otimes\gamma)\gamma+(R\otimes I\otimes \alpha)(\tau\gamma\otimes I)\gamma
+(\tau\otimes I)(\alpha\otimes R\otimes I)(I\otimes\tau\gamma)\gamma\Big)\nonumber\\
&&=\lambda\Big(R(x_{1})\otimes x_{21}\otimes x_{22}+R(x_{21})\otimes x_{22}\otimes \alpha(x_{1})+R(x_{22})\otimes\alpha(x_{1})\otimes x_{21}\Big)\nonumber\\
&&=\lambda\Big(\alpha(x_{1})\otimes x_{21}\otimes x_{22}+x_{21}\otimes x_{22}\otimes \alpha(x_{1})+x_{22}\otimes\alpha(x_{1})\otimes x_{21}\Big)\nonumber.
\end{eqnarray}
The right hand side cancels by Hom-coJacobi identity. This achieves the proof.
\end{proof}

\begin{prop}
 Let $(C, \Delta, \gamma,\alpha)$ be a post-Hom-Lie coalgebra and $(C', \Delta',\alpha')$ a cocommutative Hom-coalgebra. Then, $C'\otimes C$ is a post-Hom-Lie
coalgebra with
\begin{eqnarray}
 \tilde\Delta(a\otimes x)&=&(I\otimes \tau\otimes I)(\Delta'(a)\otimes\Delta(x))=a_{(1)}\otimes x_{(1)}\otimes a_{(2)}\otimes x_{(2)},\nonumber\\
\tilde\gamma(a\otimes x)&=&(I\otimes \tau\otimes I)(\Delta'(a)\otimes\gamma(x))=a_{(1)}\otimes x_{1}\otimes a_{(2)}\otimes x_{2},\nonumber\\
\tilde\alpha(a\otimes  x)&=&\alpha'(a)\otimes\alpha(x) \nonumber
\end{eqnarray}
for any $a\in C', x\in C$.
\end{prop}
\begin{proof} For all $a\in C', x\in C$,\\
Skew-cosymmetry :
\begin{eqnarray}
 \tilde\gamma(a\otimes x)
&=&-(I\otimes\tau\otimes I)(\tau\Delta(a)\otimes\tau\gamma(x))=-(I\otimes\tau\otimes I)(a_{(2)}\otimes a_{1}\otimes x_{(2)}\otimes x_{1})\nonumber\\
&=&-a_{(2)}\otimes x_{2}\otimes a_{(1)}\otimes x_{1}=-\eta(a_{(1)}\otimes x_{1}\otimes a_{(2)}\otimes x_{2})\nonumber\\
&=&-\eta\tilde\gamma(a\otimes x)\nonumber.
\end{eqnarray}
Hom-CoJacobi identity :
\begin{eqnarray}
 ({\bf \tilde\alpha}\otimes\tilde\gamma)\tilde\gamma(a\otimes x)&=&({\bf \tilde\alpha}\otimes\tilde\gamma)(a_{(1)}\otimes x_{1}\otimes a_{(2)}\otimes x_{2})\nonumber\\
&=&\alpha'(a_{(1)})\otimes \alpha(x_{1})\otimes\tilde\gamma(a_{(2)}\otimes x_{2})\nonumber\\
&=&\alpha'(a_{(1)})\otimes \alpha(x_{1})\otimes a_{(2)(1)}\otimes x_{21}\otimes a_{(2)(2)}\otimes x_{22},\nonumber
\end{eqnarray}
and
\begin{eqnarray}
 &&({\bf \tilde\alpha}\otimes\tilde\gamma)\tilde\gamma(a\otimes x)+\xi({\bf \tilde\alpha}\otimes\tilde\gamma)\tilde\gamma(a\otimes x)
+\xi^2({\bf \tilde\alpha}\otimes\tilde\gamma)\tilde\gamma(a\otimes x)=\nonumber\\
&&\qquad=\alpha'(a_{(1)})\otimes\alpha(x_{1})\otimes a_{(2)(1)}\otimes x_{21}\otimes a_{(2)(2)}\otimes x_{22}
+a_{(2)(1)}\otimes x_{21}\otimes a_{(2)(2)}\otimes x_{22}\otimes \alpha'(a_{(1)})\otimes \alpha(x_{1})\nonumber\\
&&\qquad\quad+a_{(2)(2)}\otimes x_{22}\otimes \alpha'(a_{(1)})\otimes\alpha(x_{1})\otimes a_{(2)(1)}\otimes x_{21}\nonumber\\
&&\qquad=({\bf I}\otimes \tau\otimes{\bf I})(I\otimes\tau\otimes\tau\otimes I)
\alpha'(a_{(1)})\otimes a_{(2)(1)}\otimes a_{(2)(2)}\otimes\Big(\alpha(x_{1})\otimes x_{21}\otimes x_{22}\nonumber\\
&&\qquad\quad+x_{21}\otimes x_{22}\otimes x_{1}+x_{22}\otimes x_{1}\otimes x_{21}\Big).\nonumber
\end{eqnarray}
Which cancels by Hom-coJacobi identity in $C$.\\
Let us verify axiom (\ref{dplc3}) for $\tilde\Delta$ and $\tilde\gamma$ :
\begin{eqnarray}
&&\qquad ({\bf (\tilde\alpha}\otimes\tilde\Delta)\tilde\gamma+(\eta\otimes{\bf I} )(\tilde\Delta\otimes {\bf \tilde\alpha})\tilde\gamma
-(\eta\otimes {\bf I})({\bf \tilde\alpha}\otimes\tilde\gamma)\tilde\alpha))(a\otimes x)=\nonumber\\
&&=({\bf \tilde\alpha}\otimes\tilde\Delta)(a_{(1)}\otimes x_{1}\otimes a_{(2)}\otimes x_{2})
+(\eta\otimes {\bf I})(\tilde\Delta\otimes {\bf \tilde\alpha})(a_{(1)}\otimes x_{1}\otimes a_{(2)}\otimes x_{2})\nonumber\\
&&\quad-(\eta\otimes {\bf I})({\bf \tilde\alpha}\otimes\tilde\gamma)(a_{(1)}\otimes x_{1}\otimes a_{(2)}\otimes x_{2})\nonumber\\
&&=\alpha'(a_{(1)})\otimes \alpha(x_{1})\otimes\tilde\Delta(a_{(2)}\otimes x_{2})
+(\tau\otimes {\bf I})(\tilde\Delta(a_{(1)}\otimes x_{1})\otimes \alpha'(a_{(2)})\otimes \alpha(x_{2}))\nonumber\\
&&\quad-(\eta\otimes {\bf I})(\alpha'(a_{(1)})\otimes \alpha(x_{1})\otimes\gamma(a_{(2)}\otimes x_{2}))\nonumber\\
&&=\alpha'(a_{(1)})\otimes\alpha(x_{1})\otimes a_{(2)(1)}\otimes x_{2(1)}\otimes a_{(2)(2)}\otimes x_{2(2)}
+a_{(1)(2)}\otimes x_{1(2)}\otimes a_{(1)(1)}\otimes x_{1(1)}\otimes \alpha'(a_{(2)})\otimes \alpha(x_{2})\nonumber\\
&&\quad-a_{(2)(1)}\otimes x_{(2)1}\otimes \alpha'(a_{(1)})\otimes\alpha(x_{(1)})\otimes a_{(2)(2)}\otimes x_{(2)(2)}\nonumber\\
&&=({\bf I}\otimes \tau\otimes{\bf I})(I\otimes\tau\otimes\tau\otimes I)
\alpha'(a_{(1)})\otimes a_{(2)(1)}\otimes a_{(2)(2)}\otimes\Big(\alpha(x_{1})\otimes \otimes x_{2(1)}\otimes \otimes x_{2(2)}\nonumber\\
&&\quad+x_{1(2)}\otimes x_{1(1)}\otimes \alpha(x_{2})-x_{(2)1}\otimes\alpha(x_{(1)})\otimes x_{(2)(2)}\Big).\nonumber
\end{eqnarray}
The right hand side vanishes by (\ref{dplc3a}).\\
Now let us verify axiom (\ref{dplc4}) for $\tilde\Delta$ and $\tilde\gamma$ :\\
As,
\begin{eqnarray}
 as_{\tilde\Delta}(a\otimes x)
&=&(\tilde\Delta\otimes \tilde\alpha)\tilde\Delta(a\otimes x)-(\tilde\alpha\otimes\tilde\Delta)\tilde\Delta(a\otimes x)\nonumber\\
&=&\tilde\Delta(a_{(1)}\otimes x_{(1)})\otimes \alpha'(a_{(2)})\otimes \alpha(x_{(2)})-\alpha'(a_{(1)})\otimes \alpha(x_{(1)})\otimes\tilde\Delta(a_{(2)}\otimes x_{2})\nonumber\\
&=&a_{(1)(1)}\otimes x_{(1)(1)}\otimes a_{(1)(2)}\otimes x_{(1)(2)}\otimes\alpha'(a_{(2)})\otimes \alpha(x_{(2)})\nonumber\\
&&-\alpha(a_{(1)})\otimes\alpha(x_{(1)})\otimes a_{(2)(1)}\otimes x_{(2)(1)}\otimes a_{(2)(2)}\otimes x_{(2)(2)}\nonumber,
\end{eqnarray}
and
\begin{eqnarray}
 (\tilde\gamma\otimes \tilde\alpha)\tilde\Delta(a\otimes x)
=a_{(1)(1)}\otimes x_{1(1)}\otimes a_{(1)(2)}\otimes x_{(1)2}\otimes\alpha'(a_{(2)})\otimes \alpha(x_{(2)}),\nonumber
\end{eqnarray}

it comes
\begin{eqnarray}
 &&\qquad(\tilde\gamma\otimes \tilde\alpha)\tilde\Delta(a\otimes x)-(\eta\otimes I)(\tilde\gamma\otimes \tilde\alpha)\tilde\Delta(a\otimes x)=\nonumber\\
&&=a_{(1)(1)}\otimes x_{(1)(1)}\otimes a_{(1)(2)}\otimes x_{(1)(2)}\otimes \alpha'(a_{(2)})\otimes \alpha(x_{(2)})\nonumber\\
&&\quad\quad-\alpha'(a_{(1)})\otimes \alpha(x_{(1)})\otimes a_{(2)(1)}\otimes x_{(2)(1)}\otimes a_{(2)(2)}\otimes x_{(2)(2)}\nonumber\\
&&\quad-a_{(1)(2)}\otimes x_{(1)(2)}\otimes a_{(1)(1)}\otimes x_{(1)(1)}\otimes \alpha'(a_{(2)})\otimes \alpha(x_{(2)})\nonumber\\
&&\quad\quad-a_{(2)(1)}\otimes x_{(2)(1)}\otimes \alpha'(a_{(1)})\otimes \alpha(x_{(1)})\otimes a_{(2)(2)}\otimes x_{(2)(2)}\nonumber\\
&&\quad+a_{(1)(1)}\otimes x_{(1)1}\otimes a_{(1)(2)}\otimes x_{(1)2}\otimes \alpha'(a_{(2)})\otimes \alpha(x_{(2)})\nonumber\\
&&=({\bf I}\otimes\tau\otimes {\bf I})(I\otimes\tau\otimes\tau\otimes I)a_{(1)(1)}\otimes a_{(1)(2)}\otimes \alpha'(a_{(2)})\otimes\Big( 
x_{(1)(1)}\otimes x_{(1)(2)}\otimes \alpha(x_{(2)})\nonumber\\
&&\quad-\alpha(x_{(1)})\otimes x_{(2)(1)}\otimes x_{(2)(2)}-x_{(1)(2)}\otimes x_{(1)(1)}\otimes \alpha(x_{(2)})
-x_{(2)(1)}\otimes \alpha(x_{(1)})\otimes x_{(2)(2)}\nonumber\\
&&\quad+x_{(1)1}\otimes x_{(1)2}\otimes \alpha(x_{(2)})\Big).\nonumber
\end{eqnarray}
The right hand side vanishes by (\ref{dplc4a})
\end{proof}
\subsection{Connexion with Hom-tridendriform coalgebras}
\begin{defn}\label{chd}
 A cocommutative Hom-tridendriform coalgebra is a quadruple $(T, \Delta_\star, \Delta_\cdot, \alpha)$ in which $(T, \Delta_\cdot, \alpha)$
 is a cocommutative Hom-coassociative coalgebra and
  $\Delta_\star : T\rightarrow T\otimes T$ is a linear operation such that :
\begin{eqnarray}
( (\Delta_\star + \tau\circ\Delta_\star +\Delta_\cdot)\otimes\alpha )\circ\Delta_\star=(\alpha\otimes\Delta_\star)\circ\Delta_\star,\label{d3}
\end{eqnarray}
\begin{eqnarray}
 (\Delta_\star\otimes\alpha)\circ\Delta_\cdot =(\alpha\otimes\Delta_\cdot)\circ\Delta_\star.\label{d4}
\end{eqnarray}
\end{defn}
\begin{prop}
 Let $(T, \Delta_{-1}, \Delta_1, \Delta_0, \alpha)$ be a Hom-tridendriform coalgebra. Then, $(T, \gamma, \Delta_\ast, \alpha)$ is a post-Hom-Lie coalgebra, where
$\gamma=(1-\tau)\circ\Delta_0$ and $\Delta_\ast= \Delta_1-\tau\circ\Delta_{-1}$.  
\end{prop}
\begin{proof}
$(T, \gamma, \alpha)$ is a Hom-Lie coalgebra, by Lemma \ref{Am1}. Now, we have:
\begin{eqnarray}
 &&\epsilon^2\circ(\alpha\otimes\gamma)\circ\Delta_\ast-\epsilon^2\circ(\Delta_\ast\otimes\alpha)\circ\gamma-(I\otimes\tau)\circ(\alpha\otimes\Delta_\ast)\circ\gamma\nonumber\\
&&=\epsilon^2\circ(\alpha\otimes\Delta_0)\circ\Delta_1-\epsilon^2\circ(I\otimes\tau)\circ(\alpha\otimes\Delta_0)\circ\Delta_1-(\Delta_0\otimes\alpha)\circ\Delta_{-1}\nonumber\\
&&\quad+(\tau\otimes I)\circ(\Delta_0\otimes\alpha)\circ\Delta_{-1}-\epsilon^2\circ(\Delta_1\otimes_\alpha)\circ\Delta_0 +(I\otimes\tau)\circ(\Delta_{-1}\otimes\alpha)\circ\Delta_0\nonumber\\
&&\quad+\epsilon\circ(\alpha\otimes\Delta_1)\circ\Delta_0-(\tau\otimes I)\circ(\alpha\otimes\Delta_{-1})-(I\otimes\tau)\circ(\alpha\otimes\Delta_1)\circ\Delta_0\nonumber\\
&&\quad +(\alpha\otimes\Delta_{-1})\circ\Delta_0 + (I\otimes\tau)\circ\epsilon^2\circ(\Delta_1\otimes\alpha)\circ\Delta_0-\epsilon\circ(\Delta_{-1}\otimes\alpha)\circ\Delta_0.
\end{eqnarray}
The left hand side vanishes by axioms in Definition \ref{tdd}.\\
Next
 \begin{eqnarray}
  &&(\alpha\otimes\Delta_\ast)\circ\Delta_\ast-(\tau\otimes I)\circ(\alpha\otimes\Delta_\ast)\circ\Delta_\ast+(\tau\otimes I)\circ(\Delta_\ast\otimes\alpha)\circ\Delta_\ast\nonumber\\
&&\quad -(\Delta_\ast\otimes\alpha)\circ\Delta_\ast+(\tau\otimes I)\circ(\gamma\otimes\alpha)\circ\Delta_\ast\nonumber\\
&&=(\alpha\otimes\Delta_1)\circ\Delta_1-(I\otimes\tau)\circ(\alpha\otimes\Delta_{-1})\circ\Delta_1-\epsilon\circ(\Delta_1\otimes\alpha)\circ\Delta_{-1}\nonumber\\
&&\quad+(I\otimes\tau)\circ\epsilon^2\circ(\Delta_{-1}\otimes\alpha)\circ\Delta_{-1}-(\tau\otimes I)\circ(\alpha\otimes\Delta_1)\circ\Delta_1\nonumber\\
&&\quad-\epsilon\circ(\alpha\otimes\Delta_{-1})\circ\Delta_1-(I\otimes\tau)\circ(\Delta_1\otimes\alpha)\circ\Delta_{-1}+\epsilon^2\circ(\Delta_{-1}\otimes\alpha)\circ\Delta_\dashv\nonumber\\
&&\quad(\tau\otimes I)\circ(\Delta_1\otimes\alpha)\circ\Delta_1-(\Delta_{-1}\otimes\otimes\alpha)\circ\Delta_1+\epsilon^2\circ(\alpha\otimes\Delta_{-1})\circ\Delta_{-1}\nonumber\\
&&\quad-(I\otimes\tau)\circ\epsilon^2\circ(\alpha_\otimes\Delta_1)\circ\Delta_{-1}-(\Delta_1\otimes\alpha)\circ\Delta_1+(\tau\otimes I)\circ(\Delta_{-1}\otimes\alpha)\circ\Delta_1\nonumber\\
&&\quad + \epsilon^2\circ(\alpha\otimes\Delta_1\circ\Delta_{-1}-(I\otimes\tau)\circ\epsilon^2\circ(\alpha\otimes\Delta_{-1})\circ\Delta_{-1}-(\Delta_0\otimes\alpha)\circ\Delta_1\nonumber\\
&&\quad+(\tau\otimes I)\circ(\Delta_0\otimes\alpha)\circ\Delta_1-(I\otimes\tau)\circ\epsilon^2\circ(\alpha\otimes\Delta_0)\circ\Delta_{-1}+\epsilon^2\circ(\alpha\otimes\Delta_0)\circ\Delta_{-1}.
 \end{eqnarray}
Which vanishes by axioms in Definition \ref{tdd}.
\end{proof}
\begin{prop}
 Let $(A, \Delta_0, \alpha, R)$ be a Hom-coassociative Rota-Baxter coalgebra. Define
\begin{eqnarray}
\Delta_{-1} &:=&(I\otimes B)\circ\Delta_0 -\Delta_0,\nonumber\\
\Delta_1 &:=&(B\otimes I)\circ\Delta_0 +\Delta_0\nonumber.
\end{eqnarray}
 Then,
$(A, \Delta_{-1}, \Delta_1, \alpha)$ is a Hom-dendriform coalgebra.
\end{prop}
\begin{proof}
 We have :
\begin{eqnarray}
 (\Delta_{-1}\otimes\alpha)\circ\Delta_{-1}
&=&(I\otimes B\otimes I)\circ(\Delta_0\otimes\alpha)\circ\Delta_{-1}-(\Delta_0\otimes\alpha)\circ\Delta_{-1}\nonumber\\
&=&(I\otimes B\otimes\alpha)\circ(\Delta_0\otimes B)\circ\Delta_0-(I\otimes B\otimes I)\circ(\Delta_0\otimes\alpha)\circ\Delta_0\nonumber\\
&&-(I\otimes I\otimes\alpha)\circ(\Delta_0\otimes B)\circ\Delta_0+(\Delta_0\otimes\alpha)\circ\Delta_0\nonumber.
\end{eqnarray}
Using the Hom-associativity and the fact that $\alpha$ commutes with $B$,
\begin{eqnarray}
 (\Delta_{-1}\otimes\alpha)\circ\Delta_{-1}
&=&(I\otimes B\otimes B)\circ(\alpha\otimes\Delta_0)\circ\Delta_0-(I\otimes B\otimes I)\circ(\alpha\otimes\Delta_0)\circ\Delta_0\nonumber\\
&&-(I\otimes I\otimes B)\circ(\alpha\otimes\Delta_0)\circ\Delta_0+(\alpha\otimes\Delta_0)\circ\Delta_0\nonumber\\
&=&\Big( (I\otimes B\otimes B)-(I\otimes B\otimes I)-(I\otimes I\otimes B)+ 1 \Big)\circ(\alpha\otimes\Delta_0)\circ\Delta_0\nonumber.
\end{eqnarray}
$B$ being a Rota-Baxter operator and adding $(1-\tau)\circ\Delta_0\circ(1+B)$,
\begin{eqnarray}
(\Delta_{-1}\otimes\alpha)\circ\Delta_{-1}&=&\Big[\Big(I\otimes(1-\tau)\circ\Delta\circ(1+B)-(I\otimes B\otimes I)-(I\otimes I\otimes B)\Big)\circ(\alpha\otimes\Delta_0)\nonumber\\
&&\quad(I\otimes \Delta_0)\circ(\alpha\otimes B)-(I\otimes(\tau\circ\Delta_0))\circ(\alpha\otimes B)\Big]\circ\Delta_0\nonumber\\
&=&\Big(I\otimes((I\otimes B)\circ\Delta_0)\circ(\alpha\otimes B)-(I\otimes I\otimes B)\circ(\alpha\otimes\Delta_0)\nonumber\\
&&+(I\otimes(B\otimes I)\circ\Delta_0)\circ(\alpha\otimes B)-(I\otimes B\otimes I)\circ(\alpha\otimes\Delta_0)\Big)\circ\Delta_0\nonumber\\
&=&(I\otimes I\otimes B+I\otimes B\otimes I)\circ(\alpha\otimes\Delta_0)\circ\Delta_{-1}\nonumber\\
&=&(\alpha\otimes\Delta_{-1})\circ\Delta_{-1}+(\alpha\otimes\Delta_{-1})\circ\Delta_{-1}.\nonumber
\end{eqnarray}
The others relations are proved similarly.
\end{proof}
In order to study the relationship between post-Hom-Lie coalgebra and cocommutative Hom-tridendriform coalgebra, we have the following definition.
\begin{defn}\label{hppad}
 A post-Hom-Poisson coalgebra is a vector space $P$ equipped with four linear maps
$\gamma : P\rightarrow P\otimes P$, $\Delta_\cdot : P\rightarrow P\otimes P$, $\Delta_\star :  P\rightarrow P\otimes P$ and $\Delta_\ast :  P\rightarrow P\otimes P$,
 and a linear map $\alpha : P\rightarrow P$ such that $(P, \gamma, \Delta_\cdot, \alpha)$ is a post-Hom-Lie coalgebra, $(P, \Delta_\star, \Delta_\ast, \alpha)$
is a cocommutative Hom-tridendriform coalgebra, and they are compatible in the sense that, the following equations are satisfied:
\begin{eqnarray}
 (\alpha\otimes\Delta_\ast)\circ\gamma=(\gamma\otimes\alpha)\circ\Delta_\ast +(\tau\otimes I)\circ(\alpha\otimes\gamma)\circ\Delta_\ast,\label{p1}
\end{eqnarray}
\begin{eqnarray}
 (I\otimes\tau)\circ(\alpha\otimes\Delta_\star)\circ\gamma=\epsilon^2\circ(\alpha\otimes\gamma)\circ\Delta_\star-\epsilon\circ(\alpha\otimes\Delta_\cdot)\circ\Delta_\ast,\label{p2}
\end{eqnarray}
\begin{eqnarray}
 (\alpha\otimes\Delta_\ast)\circ\Delta_\cdot=(\Delta_\cdot\otimes\alpha)\circ\Delta_\ast+(\tau\otimes I)\circ(\alpha\otimes\Delta_\cdot)\circ\Delta_\ast,\label{p3}
\end{eqnarray}
\begin{eqnarray}
&&\epsilon\circ(\Delta_\star\otimes\alpha)\circ\Delta_\cdot + \epsilon^2\circ(I\otimes\tau)\circ(\Delta_\star\otimes\alpha)\circ\Delta_\cdot+\epsilon\circ(\Delta_\ast\otimes\alpha)\circ\Delta_\cdot
\cr
&&=(\Delta_\cdot\otimes\alpha)\circ\Delta_\ast+(\tau\otimes I)\circ(\alpha\otimes\Delta_\cdot)\circ\Delta_\ast,\label{p4}
\end{eqnarray}
\begin{eqnarray}
&&(I\otimes\tau)\circ(\alpha\otimes\Delta_\star)\circ\Delta_\cdot=\epsilon^2\circ(\alpha\otimes\Delta_\cdot)\circ\Delta_\star+(I\otimes\tau)\circ(\Delta\otimes\alpha)\circ\Delta_\star\cr
&&-\epsilon^2\circ(\Delta_\cdot\otimes\alpha)\circ\Delta_\star+(I\otimes\tau)\circ(\gamma\otimes\alpha)\circ\Delta_\star.\label{p5}
\end{eqnarray}
\end{defn}
\begin{exam}
 A post-Poisson coalgebra is a post-Hom-Poisson coalgebra with $\alpha=Id$.
\end{exam}

\begin{rem}
Observe that $(P, \gamma, \ast, \alpha)$ is a cocommutative Hom-Poisson coalgebra \cite{DY1} i.e.
$(P, \gamma, \alpha)$ is a Hom-Lie coalgebra, $(P, \Delta_\ast, \alpha)$ is a cocommutative Hom-coassociative coalgebra and,
$\gamma$ and $\Delta_\ast$ are compatible i.e. (\ref{p1}) holds.
\end{rem}
\begin{thm}
 Let $(P, \gamma, \Delta_\cdot, \Delta_\star, \Delta_\ast, \alpha)$ be a post-Hom-Poisson coalgebra. Define two new linear operations
\begin{eqnarray}
\Delta_A(x)=(1-\tau)\circ\Delta_\cdot(x) +\gamma(x), \quad \Delta_R(x)=(1+\tau)\circ\Delta_\star(x)+\Delta_\ast(x), \nonumber
\end{eqnarray}
for any $x\in P$. Then $(P, \Delta_A, \Delta_R, \alpha)$ is a cocommutative Hom-Poisson coalgebra.
\end{thm}
\begin{proof}
\begin{enumerate}
 \item Let us first
prove that the coproduct $\Delta_A$ is a Hom-Lie cobracket :
 It is clear that it is skew-cocommutative and we have:
\begin{eqnarray}
 (\alpha\otimes\Delta_A)\circ\Delta_A&=&(\alpha\otimes\Delta_\cdot)\circ\Delta_A-(I\otimes\tau)\circ(\alpha\otimes\Delta_\cdot)\circ\Delta_A+(\alpha\otimes\gamma)\circ\Delta_A\nonumber\\
&=&(\alpha\otimes\Delta_\cdot)\circ\Delta_\cdot-(I\otimes\tau)\circ(\alpha\otimes\Delta_\cdot)\circ\Delta_\cdot+(\alpha\otimes\gamma)\circ\Delta_\cdot\nonumber\\
&&-\epsilon\circ(\Delta_\cdot\otimes\alpha)\circ\Delta_\cdot+\epsilon^2\circ(I\otimes\tau)\circ(\Delta_\cdot\otimes\alpha)\circ\Delta_\cdot-\epsilon\circ(\gamma\otimes\alpha)\circ\Delta_\cdot\nonumber\\
&&(\alpha\otimes\Delta_\cdot)\circ\gamma-(I\otimes\tau)\circ(\alpha\otimes\Delta_\cdot)\circ\gamma+(\alpha\otimes\gamma)\circ\gamma.
\end{eqnarray}
From Definition \ref{hplad}, it follows that $(1+\xi+\xi^2)\circ(\alpha\otimes\Delta_A)\circ\Delta_A=0$.

\item Next prove that the product $\Delta_R$ is cocommutative and coassociative :
The cocommutativity of the product $\Delta_R$ is trivial and we have:
 \begin{eqnarray}
(\Delta_R\otimes\alpha)\circ\Delta_R
&=&\Big(((1+\tau)\circ\Delta_\star+\Delta_\ast)\otimes\alpha\Big)\circ\Delta_R\nonumber\\
&=&((\Delta_\star+\tau\circ\Delta_\star+\Delta_\ast)\otimes\alpha)\circ\Delta_\star+(\alpha\otimes(\Delta_\star+\tau\circ\Delta_\star+\Delta_\ast))\circ\Delta_\star\nonumber\\
&&+ ((\Delta_\star+\tau\circ\Delta_\star+\Delta_\ast)\otimes\alpha)\circ\Delta_\ast\nonumber\\
&=&((\Delta_\star+\tau\circ\Delta_\star+\Delta_\ast)\otimes\alpha)\circ\Delta_\star+\epsilon^2\circ(\alpha\otimes\Delta_\star)\circ\Delta_\star+\epsilon^2\circ(\alpha\otimes\Delta_\ast)\circ\Delta_\star\nonumber\\
&&+\epsilon^2\circ(I\otimes\tau)\circ(\alpha\otimes\Delta_\star)\circ\Delta_\star+(1+\tau\otimes I)\circ(\Delta_\star\otimes\alpha)\circ\Delta_\ast+(\Delta_\ast\otimes\alpha)\circ\Delta_\ast\nonumber.
 \end{eqnarray}
Using axioms in Definition \ref{chd}, one has :
\begin{eqnarray}
 (\Delta_R\otimes\alpha)\circ\Delta_R
&=&(1+(I\otimes\tau))\circ(\alpha\otimes\Delta_\star)\circ\Delta_\star+(\epsilon^2\circ(I\otimes\tau)+\epsilon)\circ(\Delta_\star\otimes\alpha)\circ\Delta_\star\nonumber\\
&&+\epsilon^2\circ(I\otimes\tau)\circ(\Delta_\ast\otimes\alpha)\circ\Delta_\star+(I\otimes\tau)\circ(\alpha\otimes\Delta_\star)\circ\Delta_\ast+(\alpha\otimes\Delta_\ast)\circ\Delta_\star\nonumber\\
&&+(\alpha\otimes\Delta_\star)\circ\Delta_\ast+(\alpha\otimes\Delta_\ast)\circ\Delta_\ast\nonumber\\
&=&(1+I\otimes\tau)\circ(\alpha\otimes\Delta_\star)\circ\Delta_\star+(\alpha\otimes\Delta_\ast)\circ\Delta_\star+(\epsilon^2\circ(I\otimes\tau)+\epsilon)\circ(\Delta_\star\otimes\alpha)\circ\Delta_\star\nonumber\\
&&+\epsilon^2\circ(I\otimes\tau)\circ(\Delta_\ast\otimes\alpha)\circ\Delta_\star+(I\otimes\tau+1)\circ(\alpha\otimes\Delta_\star)\circ\Delta_\ast+(\alpha\otimes\Delta_\ast)\circ\Delta_\ast\nonumber\\
&=&(\alpha\otimes\Delta_R)\circ\Delta_R.\nonumber
\end{eqnarray}
\item
Finally, let us prove (\ref{p1}) for $\Delta_A$ and $\Delta_R$.
\begin{eqnarray}
 &&\quad(\alpha\otimes\Delta_R)\circ\Delta_A-(\Delta_A\otimes\alpha)\circ\Delta_R-(\tau\otimes I)\circ(\alpha\otimes\Delta_A)\circ\Delta_R\nonumber\\
&&=(1+I\otimes\tau)\circ(\alpha\otimes\Delta_\star)\circ\Delta_A+(\alpha\otimes\Delta_A)\circ\Delta_A+(\tau\otimes I-1)\circ(\Delta_\cdot\otimes\alpha)\circ\Delta_R\nonumber\\
&&\quad-(\gamma\otimes\alpha)\circ\Delta_R-(\tau\otimes I)\circ(\alpha\otimes\Delta_\cdot)\circ\Delta_R\nonumber\\
&&=(1+I\otimes\tau)\circ(\alpha\otimes\Delta_\star)\circ\Delta_\cdot+(\alpha\otimes\Delta_\ast)\circ\Delta_\cdot-(\epsilon+\epsilon^2\circ(I\otimes\tau))\circ(\Delta_\star\otimes\alpha)\circ\Delta_\cdot\nonumber\\
&&-\epsilon\circ(\Delta_\ast\otimes\alpha)\circ\Delta_\cdot+(1+I\otimes\tau)\circ(\alpha\otimes\Delta_\star)\circ\gamma+(\alpha\otimes\Delta_\ast)\circ\gamma\nonumber\\
&&+(\tau\otimes I-1)\circ(\Delta_\cdot\otimes\alpha)\circ\Delta_\star-(\gamma\otimes\alpha)\circ\Delta_\star\nonumber.
\end{eqnarray}
The left hand side vanishes by expanding and using once axioms (\ref{p1}), (\ref{p3}) and (\ref{p4}),
 twice  relations (\ref{p2}) and (\ref{p5}).
\end{enumerate}

Therefore $(P, \Delta_\ast, \Delta_A, \alpha)$ is a cocommutative Hom-Poisson coalgebra.
\end{proof}
\section{Post-Hom-Lie comodules}
 This last section is devoted to the  introduction of comodules over post-Hom-Lie coalgebras and the given of various constructions by twisting.
\begin{defn}
 A Hom-module is a pair $(M,\alpha_M)$ in which $M$ is a vector space and $\alpha_M: M\longrightarrow M$ is a linear map.
\end{defn}
\begin{defn}
Let $(L, \gamma, \Delta_\cdot, \alpha)$ be a post-Hom-Lie coalgebra. A comodule over $L$ is a Hom-module $(M, \alpha_M)$ equipped with two
linear maps $\Delta_\diamond :  M\rightarrow L\otimes M$ and $\Delta_\bullet : M\rightarrow L\otimes M$, called structure maps, such that :
 \begin{eqnarray}
  \Delta_\diamond\circ\alpha_M&=&(\alpha_M\otimes\alpha)\circ\Delta_\diamond\quad\mbox{and}\quad\Delta_\bullet\circ\alpha_M=(\alpha_M\otimes\alpha)\circ\Delta_\bullet,\label{ma1}\\
(\gamma\otimes\alpha_M)\circ\Delta_\diamond&=&(\alpha\otimes\Delta_\diamond)\circ\Delta_\diamond-(\tau\otimes Id_M)\circ(\alpha\otimes\Delta_\diamond)\circ\Delta_\diamond,\label{ma2}\\
(\Delta_\cdot\otimes\alpha_M)\circ\Delta_\diamond&=&(\alpha\otimes\Delta_\diamond)\circ\Delta_\bullet-(\tau\otimes Id_M)\circ(\alpha\otimes\Delta_\bullet)\circ\Delta_\bullet,\label{ma3},\\
(\gamma\otimes\alpha_M)\circ\Delta_\bullet&=&(1-(\tau\otimes Id_M))\circ(\alpha\otimes\Delta_\bullet)\circ\Delta_\bullet+((\tau\otimes Id_M)-1)\circ(\Delta_\cdot\otimes\alpha_M)\circ\Delta_\bullet\label{ma4}.
\end{eqnarray}

\end{defn}
\begin{exam}
 Any post-Hom-Lie coalgebra is a comodule over itself.
\end{exam}

\begin{prop}
 Let $(M_i, \Delta_{\diamond i}, \Delta_{\bullet i}, \alpha_{M_i})$ ($i=1, 2$) be two comodules over the post-Hom-Lie coalgebra $(L, \gamma, \Delta_\cdot, \alpha)$.
Then, the direct sum $M=M_1\oplus M_2$ is a comodule over $L$ for the structure maps
$$\Delta_\diamond=\Delta_{\diamond_1}\oplus\Delta_{\diamond_2}, \quad \Delta_\bullet=\Delta_{\bullet_1}\oplus\Delta_{\bullet_2},\quad\alpha_M=\alpha_{M_1}\oplus\alpha_{M_2}.$$
\end{prop}
\begin{proof}
It is straightforward by calculation.
\end{proof}
\begin{prop}\label{nph}
  Let $(L, \gamma, \Delta_\cdot, \alpha)$ be a comultiplicative post-Hom-Lie coalgebra. Define two new operations $\Delta_\circ, \Delta_\ast : L\rightarrow L\otimes L$
by
$$\Delta_\circ:=(\alpha^k\otimes I)\circ\gamma\quad\mbox{and}\quad \Delta_\ast :=(\alpha^k\otimes I)\circ\Delta_\cdot.$$
Then, $(L, \Delta_\circ, \Delta_\ast, \alpha)$ is a comodule over the post-Hom-Lie coalgebra $(L, \gamma, \Delta_\cdot, \alpha)$.
\end{prop}
\begin{proof}
 We have:
\begin{eqnarray}
 (\Delta_\cdot\otimes\alpha)\circ\Delta_\circ &=&(\Delta_\cdot\otimes I)\circ(\alpha^k\otimes\alpha)\circ\gamma\nonumber\\
&=&(I\otimes\alpha^k\otimes I)\circ(\alpha^{k+1}\otimes\gamma)\circ\Delta_\cdot-(\tau\otimes I)\circ(I\otimes\alpha^k\otimes I)\circ(\alpha^{k+1}\otimes\Delta)\circ\gamma\nonumber\\
&=&(I\otimes\alpha^k\otimes I)\circ(\alpha\otimes\gamma)\circ\Delta_\ast-(\tau\otimes I)\circ(I\otimes\alpha^k\otimes I)\circ(\alpha\otimes\Delta_\cdot)\circ\Delta_\circ\nonumber\\
&=&(\alpha\otimes\Delta_\circ)\circ\Delta_\ast-(\tau\otimes I)\circ(I\otimes\alpha^k\otimes I)\circ(\alpha\otimes\Delta_\cdot)\circ\Delta_\circ.\nonumber
\end{eqnarray}
The other axioms are proved analogously.
\end{proof}
\begin{thm}\label{mpm}
Let $(M_i, \Delta_{\diamond_i}, \Delta_{\bullet_i}, \alpha_{M_i})$ ($i=1, 2$) be two comodules over the multiplicative post-Hom-Lie coalgebra
 $(L, \gamma, \Delta_\cdot, \alpha)$.
The linear maps $\Delta_\diamond, \Delta_\bullet :  M_1\otimes M_2\rightarrow L\otimes M_1\otimes M_2$ and the linear map
 $\alpha_M : M_1\otimes M_2\rightarrow M_1\otimes M_2$ defined by
 \begin{eqnarray}
\alpha_M&:=&\alpha_{M_1}\otimes\alpha_{M_2}\nonumber\\
  \Delta_\diamond &:=&(\alpha^k\otimes Id_{M_1})\circ\Delta_{\diamond_1} \otimes\alpha_{M_2}+    \alpha_{M_1}\otimes(\alpha^k\otimes Id_{M_2})\circ\Delta_{\diamond_2} \nonumber\\
   \Delta_\bullet &:=&(\alpha^k\otimes Id_{M_1})\circ\Delta_{\bullet_1} \otimes\alpha_{M_2}+    \alpha_{M_1}\otimes(\alpha^k\otimes Id_{M_2})\circ\Delta_{\bullet_2} \nonumber.
 \end{eqnarray}
give to $M_1\otimes M_2$ an $L$-comodule structure.
\end{thm}
\begin{proof}
We first have:
\begin{eqnarray}
\Delta_\diamond\circ \alpha_M
&=&\Big((\alpha^k\otimes Id_{M_1})\circ\Delta_{\diamond_1} \otimes\alpha_{M_2}+    \alpha_{M_1}\otimes(\alpha^k\otimes Id_{M_2})\circ\Delta_{\diamond_2}  \Big)\circ\alpha_M\nonumber\\
&\stackrel{(\ref{ma1})}{=}&(\alpha^{k+1}\otimes \alpha_{M_1})\circ\Delta_{\diamond_1} \otimes\alpha^2_{M_2}+    \alpha^2_{M_1}\otimes(\alpha^{k+1}\otimes  \alpha_{M_2})\circ\Delta_{\diamond_2}\nonumber\\
&=&(\alpha\otimes\alpha_{M_1}\otimes\alpha_{M_2})\circ\Delta_\diamond\nonumber\\
&=&(\alpha\otimes\alpha_{M})\otimes\Delta_\diamond\nonumber.
\end{eqnarray}
Then,
  \begin{eqnarray}
   \Delta_\diamond\circ\Delta_\bullet
&=&(\alpha\otimes\Delta_\diamond)\circ\Delta_\bullet\nonumber\\
&=&\Big(\alpha\otimes[(\alpha^k\otimes Id_{M_1})\circ\Delta_{\diamond_1} \otimes\alpha_{M_2}+    \alpha_{M_1}\otimes(\alpha^k\otimes Id_{M_2})\circ\Delta_{\diamond_2}]\Big)\circ\Delta_\bullet\nonumber\\
&=&\Big(\alpha^{k+1}\otimes(\alpha^k\otimes Id_{M_1})\circ\Delta_{\diamond_1})\Big)\circ\Delta_{\diamond_1} \otimes\alpha^2_{M_2}\nonumber\\
&&+(\alpha^{k+1}\otimes\alpha_{M_1})\circ\Delta_{\diamond_1}\otimes(\alpha^{k+1}\otimes\alpha_{M_2})\circ\Delta_{\bullet_2}\quad(\mbox{by} \;(\ref{ma1}))\nonumber\\
&&+(\alpha^{k+1}\otimes\alpha_{M_1})\circ\Delta_{\bullet_1}\otimes(\alpha^{k+1}\otimes\alpha_{M_2})\circ\Delta_{\diamond_2}\quad(\mbox{by} \;(\ref{ma1}))\nonumber\\
&&+\alpha^2_{M_1}\otimes\Big(\alpha^{k+1}\otimes((\alpha^k\otimes Id_{M_2})\circ\Delta_{\diamond_2})\Big)\circ\Delta_{\bullet_2}\nonumber.
  \end{eqnarray}
It follows that
\begin{eqnarray}
 &&\qquad(\alpha\otimes\Delta_\diamond)\circ\Delta_\bullet-(\tau\otimes I_{M})(\alpha\otimes\Delta_\diamond)\circ\Delta_\bullet=\nonumber\\
&&=\Big(\alpha^{k+1}\otimes(\Delta_{\diamond_1}\circ(\alpha^k\otimes I_{M_1})\circ\Delta_{\bullet_1})\Big)\otimes\alpha^2_{M_2}
-\Big((\tau\otimes I_{M})\circ(\alpha^{k+1}\otimes(\Delta_{\diamond_1}\circ(\alpha^k\otimes I_{M_1})\circ\Delta_{\bullet_1})\Big)\otimes\alpha^2_{M_2}\nonumber\\
&&\;\;+\alpha^2_{M_1}\otimes\Big(\alpha^{k+1}\otimes((\alpha^k\otimes I_{M_2})\circ\Delta_{\diamond_2}\circ\Delta_{\bullet_2})\Big)
-\alpha_{M_1}^2\otimes\Big((\tau\otimes I_{M_2})\circ(\alpha^{k+1}\otimes((\alpha^k\otimes I_{M_2})\circ\Delta_{\diamond_2}\circ\Delta_{\bullet_1})\Big)
\nonumber\\
&&=((\Delta\circ\alpha^k)\otimes\alpha_{M_1})\circ\Delta_{\diamond_1}\otimes\alpha^2_{M_2}+\alpha^2_{M_1}\otimes((\Delta\circ\alpha^k)\otimes\alpha_{M_2})\circ\Delta_{\diamond_2}\nonumber\\
&&=(\Delta_\cdot\otimes_\diamond\otimes\alpha_{M_1}\otimes\alpha_{M_2})\circ\Delta_\diamond=(\Delta_\cdot\otimes\alpha_M)\circ\Delta_\diamond.\nonumber
\end{eqnarray}
The others relations are proved in a similar way.
\end{proof}
\begin{cor}
 Let $(L, \gamma, \Delta_\cdot, \alpha)$ be a comultiplicative post-Hom-Lie coalgebra. Then, $(L\otimes L, \alpha\otimes\alpha)$ is an $L$-comodule
with the actions
 \begin{eqnarray}
  \Delta_\diamond &:=&(\alpha\otimes I)\circ\gamma\otimes\alpha + (\tau\otimes I)\circ\Big(\alpha\otimes(\alpha\otimes I)\circ\gamma\Big),\nonumber\\
  \Delta_\bullet&:=&(\alpha\otimes I)\circ\Delta_\cdot\otimes\alpha + (\tau\otimes I)\circ\Big(\alpha\otimes(\alpha\otimes I)\circ\Delta_\cdot\Big)\nonumber.
 \end{eqnarray}
\end{cor}
\begin{proof}
 It comes from Proposition \ref{nph} and Theorem \ref{mpm}.
\end{proof}
\begin{thm}
 Let $(M, \Delta_\diamond, \Delta_\bullet, \alpha_M)$ be a comodule over the post-Hom-Lie coalgebra $(L, \gamma, \Delta_\cdot, \alpha)$. For any non-negative integer $n$,
 define
\begin{eqnarray}
 \Delta^{n,0}_\diamond&:=&(\alpha^n\otimes Id_M)\circ\Delta_\diamond :  M\rightarrow L\otimes M,\label{ms1}\\
\Delta^{n,0}_\bullet&:=&(\alpha^n\otimes Id_M)\circ\Delta_\bullet :  M\rightarrow L\otimes M\label{ms2}
\end{eqnarray}
Then, $(M, \Delta^{n,0}_\diamond, \Delta^{n,0}_\bullet, \alpha_M)$ is an $L$-comodule.
\end{thm}
\begin{proof}
For any positive integer $n$, we have
 \begin{eqnarray}
  \Delta^{n,0}_\diamond&=&(\alpha^n\otimes Id_M)\circ\Delta_\diamond\circ\alpha_M
=(\alpha^n\otimes Id_M)\circ(\alpha\otimes\alpha_M)\circ\Delta_\diamond\nonumber\\
&=&(\alpha\otimes \alpha_M)\circ(\alpha^n\otimes Id_M)\circ\Delta_\diamond=(\alpha\otimes\alpha_M)\circ\Delta^{n,0}_\diamond.\nonumber
 \end{eqnarray}
Next
\begin{eqnarray}
(\gamma\otimes\alpha_M)\circ \Delta^{n,0}_\diamond
&=&(\gamma\otimes\alpha_M)\circ(\alpha^n\otimes Id_M)\circ\Delta_\diamond\nonumber\\
&=&(\gamma\circ\alpha^n\otimes\alpha_M)\circ\Delta_\diamond\nonumber\\
&=&(\gamma\circ(\alpha^n\otimes\alpha^n)\otimes\alpha_M)\circ\Delta_\diamond\nonumber\\
&=&(\alpha^n\otimes\alpha^n\otimes Id_M)\circ(\gamma\otimes\alpha_M)\circ\Delta_\diamond\nonumber.
\end{eqnarray}
Then using (\ref{ma2}) and (\ref{ms1}) we obtain
\begin{eqnarray}
 (\gamma\otimes\alpha_M)\circ\Delta^{n,0}_\diamond
&=&(\alpha^n\otimes\alpha^n\otimes Id_M)\circ(1-(\tau\otimes Id_M))\circ(\alpha\otimes\Delta_\diamond)\circ\Delta_\diamond\nonumber\\
&=&\Big(\alpha^{n+1}\otimes(\alpha^n\otimes Id_M)\circ\Delta_\diamond\Big)\circ\Delta_\diamond\nonumber\\
&&- \Big(\alpha^{n+1}\otimes[(\tau\otimes Id_M)\circ(\alpha^n\otimes Id_M)\circ\Delta_\diamond]\Big)\circ\Delta_\diamond\nonumber\\
&=&(\alpha^{n+1}\otimes\Delta^{n,0}_\diamond)\circ\Delta_\diamond-(\tau\otimes Id_M)(\alpha^{n+1}\otimes\Delta^{n,0}_\diamond)\circ\Delta_\diamond\nonumber\\
&=&(\alpha\otimes\Delta^{n,0}_\diamond)\circ(\alpha^n\otimes Id_M)\circ\Delta_\diamond\nonumber\\
&&-(\tau\otimes Id_M)(\alpha\otimes\Delta^{n,0}_\diamond)\circ(\alpha^n\otimes Id_M)\circ\Delta_\diamond\nonumber\\
&=&(\alpha\otimes\Delta^{n,0}_\diamond)\circ\Delta^{n,0}_\diamond-(\tau\otimes Id_M)(\alpha\otimes\Delta^{n,0}_\diamond)\circ\Delta^{n,0}_\diamond\nonumber.
\end{eqnarray}
The others identities are similarly proved.
\end{proof}
\begin{thm}
 Let $(M, \Delta_\diamond, \Delta_\bullet, \alpha_M)$ be a comodule over the post-Hom-Lie coalgebra $(L, \gamma, \Delta_\cdot, \alpha)$. For any non-negative integer $k$,
 define
\begin{eqnarray}
 \Delta^{0, k}_\diamond&:=&\Delta_\diamond\circ\alpha_M^{2^k} :  M\rightarrow L\otimes M,\label{ms3}\\
\Delta^{0, k}_\bullet&:=&\Delta_\bullet\circ\alpha_M^{2^k} :  M\rightarrow L\otimes M.\label{ms4}
\end{eqnarray}
Then, $(M, \Delta^{0,k}_\diamond, \Delta^{0,k}_\bullet, \alpha_M^{2^k})$ is a comodule over 
$(L, \gamma\circ\alpha^{2^k-1}, \Delta_\cdot\circ\alpha^{2^k-1}, \alpha^{2^k})$.
\end{thm}
\begin{proof}
For any positif integer $k$, we have
 \begin{eqnarray}
  \Delta^{0,k}_\diamond\circ\alpha_M^{2^k}
&=&\Delta_\diamond\circ\alpha_M^{2^k-1}\circ\alpha_M^{2^k}
=\Delta_\diamond\circ\alpha_M^{2^k}\circ\alpha_M^{2^k-1}
=\Delta_\diamond\circ\alpha_M\circ\alpha_M^{2^k-1}\circ\alpha_M^{2^k-1}\nonumber\\
&=&(\alpha\otimes\alpha_M)\circ\Delta_\diamond\circ\alpha_M^{2^k-2}\circ\alpha_M^{2^k-1}
=(\alpha^2\otimes\alpha^2_M)\circ\Delta_\diamond\circ\alpha_M^{2^k-2}\circ\alpha_M^{2^k-1}\nonumber\\
&=&(\alpha^2\otimes\alpha_M^2)\circ\Delta_\diamond\circ\alpha_M^{2^k-2}\circ\alpha_M^{2^k-1}
=\dots=(\alpha^{2^k}\otimes\alpha_M^{2^k})\circ\Delta_\diamond\circ\alpha_M^{2^k-1}\nonumber\\
&=&\Delta^{0, k}_\diamond\circ(\alpha^{2^k}\otimes\alpha_M^{2^k}).\nonumber
 \end{eqnarray}
Next
\begin{eqnarray}
 &&\qquad(\gamma\circ\alpha^{2^k-1}\otimes\alpha_M^{2^k})\circ\Delta^{0, k}_\diamond=\nonumber\\
&&=(\gamma\circ\alpha^{2^k-1}\otimes\alpha_M^{2^k})\circ\Delta_\diamond\circ\alpha_M^{2^k-1}\nonumber\\
&&=(\gamma\circ\alpha^{2^k-2}\otimes\alpha_M^{2^k-1})\circ(\alpha\otimes\alpha_M)\circ\Delta_\diamond\circ\alpha_M^{2^k-1}\nonumber\\
&&=(\gamma\circ\alpha^{2^k-2}\otimes\alpha_M^{2^k-1})\circ\Delta_\diamond\circ\alpha_M^{2^k}=\dots=\nonumber\\
&&=(\gamma\otimes\alpha_M)\circ\Delta_\diamond\circ\alpha_M^{2(2^k-1)}\nonumber.
\end{eqnarray}
Then using (\ref{ma2}) and  the first identity in (\ref{ma1}) ($2^k-1$ times) we get
\begin{eqnarray}
&&\qquad(\gamma\circ\alpha^{2^k-1}\otimes\alpha_M^{2^k})\circ\Delta^{0, k}_\diamond\nonumber\\
&&=\Big((\alpha\otimes\Delta_\diamond)\circ\Delta_\circ-(\tau\otimes Id_M)\circ(\alpha\otimes\Delta_\diamond)\circ\Delta_\diamond\Big)\circ\alpha_M^{2(2^k-1)}\nonumber\\
&&=(\alpha\otimes\Delta_\diamond)\circ\Delta_\diamond\circ\alpha_M^{2^k-1}\circ\alpha_M^{2^k-1}
-(\tau\otimes Id_M)\circ(\alpha\otimes\Delta_\diamond)\circ\Delta_\diamond\circ\alpha_M^{2^k-1}\circ\alpha_M^{2^k-1}\nonumber\\
&&=(\alpha\otimes\Delta_\diamond)\circ\Delta_\diamond\circ\alpha_M\circ\alpha_M^{2^k-1}\circ\alpha_M^{2^k-1}
-(\tau\otimes Id_M)\circ(\alpha\otimes\Delta_\diamond)\circ\Delta_\diamond\circ\alpha_M\circ\alpha_M^{2^k-1}\circ\alpha_M^{2^k-1}\nonumber\\
&&=(\alpha\otimes\Delta_\diamond\otimes\alpha_M)\circ\Delta_\diamond\circ\alpha_M^{2^k-2}\circ\alpha_M^{2^k-1}
-(\tau\otimes Id_M)\circ(\alpha\otimes\Delta_\diamond\otimes\alpha_M)\circ\Delta_\diamond\circ\alpha_M^{2^k-2}\circ\alpha_M^{2^k-1}\nonumber\\
&&=\dots=\nonumber\\
&&=(\alpha^{2^k}\otimes\Delta_\diamond\circ\alpha_M^{2^k-1})\circ\Delta_\diamond\circ\alpha_M^{2^k-1}
-(\tau\otimes Id_M)\circ(\alpha^{2^k}\otimes\Delta_\diamond\circ\alpha_M^{2^k-1})\circ\Delta_\diamond\circ\alpha_M^{2^k-1}\nonumber\\
&&=(\alpha^{2^k}\otimes\Delta^{0, k}_\diamond)\circ\Delta^{0, k}_\diamond-(\tau\otimes Id_M)\circ(\alpha^{2^k}\otimes\Delta^{0, k}_\diamond)\circ\Delta^{0, k}_\diamond\nonumber.
\end{eqnarray}

The proofs of the rest of the conditions are analogue.
\end{proof}
\begin{thm}
 Let $(M, \Delta_\diamond, \Delta_\bullet, \alpha_M)$ be a comodule over the post-Hom-Lie coalgebra $(L, \gamma, \Delta_\cdot, \alpha)$.
Let $\beta : L\rightarrow L$ be an endomorphism of $A$ and $\beta_M : M\rightarrow M$ be a linear map such that
$\alpha_M\circ\beta_M=\beta_M\circ\alpha_M$, $\Delta_\diamond\circ\beta_M=(\beta\otimes\beta_M)\circ\Delta_\diamond$ and 
$\Delta_\bullet\circ\beta_M=(\beta\otimes\beta_M)\circ\Delta_\bullet$.
Define
\begin{eqnarray}
\tilde\Delta_\diamond:=(\beta\otimes Id_M)\circ\Delta_\diamond\circ\beta_M :  M\rightarrow L\otimes M,\label{ms5}\\
\tilde\Delta_\bullet:=(\beta\otimes Id_M)\circ\Delta_\bullet\circ\beta_M :  M\rightarrow L\otimes M.\label{ms6}
\end{eqnarray}
Then, $(M, \tilde\Delta_\diamond, \tilde\Delta_\bullet, \alpha_M\circ\beta_M)$ is a comodule over $(L, \gamma\circ\beta, \Delta_\cdot\circ\beta, \beta\circ\alpha)$.
\end{thm}
\begin{proof}
We have
\begin{eqnarray}
\tilde\Delta_\diamond\circ(\alpha_M\circ\beta_M)
&=&(\beta\otimes Id_M)\circ\Delta_\diamond\circ\beta_M\circ(\alpha_M\circ\beta_M)\nonumber\\
&=&(\beta\otimes Id_M)\circ(\beta\otimes\beta_M)\circ\Delta_\diamond\circ\alpha_M\circ\beta_M\nonumber\\
&=&(\beta\otimes Id_M)\circ(\beta\otimes\beta_M)\circ(\alpha\otimes\alpha_M)\circ\Delta_\diamond\circ\beta_M\nonumber\\
&=&(\beta\circ\alpha\otimes \beta_M\circ\alpha_M)\circ(\beta\otimes Id_M)\circ\Delta_\diamond\circ\beta_M\nonumber\\
&=&(\beta\circ\alpha\otimes\beta_M\circ\alpha_M)\circ\tilde\Delta_\diamond.\nonumber
\end{eqnarray}
Next
\begin{eqnarray}
 &&\qquad(\gamma\circ\beta\otimes\beta_M\circ\alpha_M)\circ\tilde\Delta_\diamond\nonumber\\
&&=(\gamma\circ\beta\otimes\alpha_M\circ\beta_M)\circ(\beta\otimes\beta_M)\circ(\beta\otimes Id_M)\circ\Delta_\circ\beta_M\nonumber\\
&&=(\gamma\otimes\alpha_M)(\beta\circ Id_M)\circ(\beta\otimes\beta_M)\circ\Delta_\circ\beta_M \nonumber\\
&&=(\beta\otimes\beta\otimes Id_M)\circ(\gamma\otimes \alpha_M)\circ(\beta\otimes Id_M)\circ(\beta\otimes\beta_M)\circ\Delta_\diamond\circ\beta_M\nonumber\\
&&=(\beta\otimes\beta\otimes Id_M)\circ(\gamma\otimes\alpha_M)\circ\Delta_\diamond \circ\beta_M^2\nonumber.
\end{eqnarray}
Now using (\ref{ma2}), the first identity in (\ref{ma1}) and (\ref{ms5}) we obtain
\begin{eqnarray}
&&\qquad(\gamma\circ\beta\otimes\beta_M\circ\alpha_M)\circ\tilde\Delta_\diamond=\nonumber\\
&&=(\beta\otimes\beta\otimes Id_M)\circ\Big((\alpha\otimes\Delta_\diamond)\circ\Delta_\diamond
-(\tau\otimes Id_M)\circ(\alpha\otimes\Delta_\diamond)\circ\Delta_\diamond\Big)\circ\beta^2_M\nonumber\\
&&=[\beta\circ\alpha\otimes(\beta\otimes Id_M)\circ\Delta_\diamond]\circ\Delta_\diamond\circ\beta_M\circ\beta_M\nonumber\\
&&\quad-(\tau\otimes Id_M)\circ[\beta\circ\alpha\otimes(\beta\otimes Id_M)\circ\Delta_\diamond]\circ\Delta_\diamond\circ\beta_M\circ\beta_M\nonumber\\
&&=[\beta\circ\alpha^2\otimes(\beta\otimes Id_M)\circ\Delta_\diamond\circ\alpha_M]\circ\Delta_\diamond\circ\beta\nonumber\\
&&\quad-(\tau\otimes Id_M)[\beta\circ\alpha^2\otimes(\beta\otimes Id_M)\circ\Delta_\diamond\circ\alpha_M]\circ\Delta_\diamond\circ\beta\nonumber\\
&&=(\beta\circ\alpha^2\otimes\tilde\Delta_\diamond)\circ\Delta_\diamond\circ\beta_M
-(\tau\otimes Id_M)(\beta\circ\alpha^2\otimes\tilde\Delta_\diamond)\circ\Delta_\diamond\circ\beta_M\nonumber\\
&&=(\beta\circ\alpha\otimes\tilde\Delta_\diamond)\circ(\beta\otimes Id_M)\circ\Delta_\diamond\circ\beta_M\nonumber\\
&&\quad-(\tau\otimes Id_M)\circ(\beta\circ\alpha\otimes\tilde\Delta_\diamond)\circ(\beta\otimes Id_M)\circ\Delta_\diamond\circ\beta_M\nonumber\\
&&=
(\beta\circ\alpha\otimes\tilde\Delta_\diamond)\circ\tilde\Delta_\diamond
-(\tau\otimes Id_M)\circ(\beta\circ\alpha\otimes\tilde\Delta_\diamond)\circ\tilde\Delta_\diamond\nonumber.
\end{eqnarray}
The others identities are analogously proved.
\end{proof}

\end{document}